\documentclass[12pt,leqno,draft]{article} 
\usepackage{amsmath,amssymb,ams	thm,amsfonts}
\usepackage{enumerate,color}
\usepackage{slashbox}
\usepackage{hhline}
\usepackage{comment}
\usepackage{tabularx}
\makeatletter
\def
\@cite
#1#2{[{{\bfseries #1}\if@tempswa , #2\fi}]}
\renewcommand{\section}{%
\@startsection{section}{1}{\z@}
{0.5truecm plus -1ex minus -.2ex}%
{1.0ex plus .2ex}{\bfseries\large}}
\def
\@seccntformat#1{\csname the#1\endcsname.\ }
\makeatother
\numberwithin{equation}{section} 
\newtheorem{thm}{Theorem}[section]
\newtheorem{corollary}[thm]{Corollary}
\newtheorem{lem}[thm]{Lemma}

\theoremstyle{definition}
\newtheorem{df}{Definition}[section]
\newtheorem{remark}{Remark}[section]

\newcommand{\ep}{\varepsilon}
\newcommand{\pa}{\partial}
\newcommand{\Rn}{\mathbb{R}^n}

\begin{document}
\footnote[0]
	{2010{\it Mathematics 
	Subject Classification\/}. 
	Primary: 35B44, 35K51, 35K59; Secondary: 
	35Q92, 92C17.
	}
\begin{center}
	\Large{{\bf 
	Effect of nonlinear diffusion on a lower bound
	\\ 
	for the blow-up time
	\\
	in a fully parabolic chemotaxis system
   }}
\end{center}
\vspace{5pt}
\begin{center}
	\hspace{-10pt}Teruto Nishino%
	\footnote[0]{{E-mail addresses}: 
		{\tt teruto.nishino@gmail.com},   
		{\tt yokota@rs.kagu.tus.ac.jp}}\\
	Tomomi Yokota%
	\footnote{Corresponding author.}
	\footnote{Partially supported by Grant-in-Aid 
		for Scientific Research (C), 
		No.\,25400119.}\\
	\vspace{12pt}
	Department of Mathematics\\
	Tokyo University of Science\\
\end{center}
\begin{center}    
	\small \today
\end{center}
\vspace{2pt}
\newenvironment{summary}
{\vspace{.5\baselineskip}\begin{list}{}{%
	\setlength{\baselineskip}{0.85\baselineskip}
	\setlength{\topsep}{0pt}
	\setlength{\leftmargin}{12mm}
	\setlength{\rightmargin}{12mm}
	\setlength{\listparindent}{0mm}
	\setlength{\itemindent}{\listparindent}
	\setlength{\parsep}{0pt}
	\item\relax}}{\end{list}
	\vspace{.5\baselineskip}}
\begin{summary}
{\footnotesize {\bf Abstract.}
This paper deals with 
a lower bound 
for the blow-up time 
for solutions 
of the fully parabolic chemotaxis system
\begin{equation*}
	\begin{cases}
		u_t=\nabla \cdot [(u+\alpha)^{m_1-1}
		\nabla u-\chi u(u+\alpha)^{m_2-2}
		\nabla v]
		& {\rm in} \; \Omega \times (0,T),
		\\[1mm]
		v_t=\Delta v-v+u
		& {\rm in} \; \Omega \times (0,T)
		\\[1mm]
	\end{cases}
\end{equation*}
under Neumann boundary conditions 
and initial conditions, 
where $\Omega$ is a general bounded domain 
in $\Rn$ with smooth boundary, 
$\alpha>0$, 
$\chi>0$, 
$m_1, m_2 \in \mathbb{R}$ 
and 
$T>0$.  
Recently, 
Anderson--Deng \cite{Anderson-Deng_2017} 
gave a lower bound for the blow-up time 
in the case that $m_1=1$
and $\Omega$ is a convex 
bounded domain. 
The purpose of this paper is to generalize 
the result in \cite{Anderson-Deng_2017} 
to the case that $m_1 \neq 1$ and $\Omega$ 
is a non-convex bounded domain. 
The key to the proof is 
to make a sharp estimate by using 
the Gagliardo--Nirenberg inequality 
and an inequality for boundary integrals.
As a consequence, the main result of this paper 
reflects the effect of nonlinear diffusion 
and 
need not assume the convexity of $\Omega$.
}
\end{summary}
\newpage
\section{Introduction}\label{Section 1}
\indent
In this paper we consider a 
lower bound for the blow-up time 
in the following 
fully parabolic chemotaxis system 
with nonlinear diffusion:                                                                                                                         
\begin{equation}\label{KS}
	\begin{cases}
		u_t=\nabla \cdot [(u+\alpha)^{m_1-1} 
		\nabla u-\chi u(u+\alpha)^{m_2-2} 
		\nabla v] 
		& 
		{\rm in} \ \Omega \times (0,T), 
		\\[1mm]
		v_t=\Delta v-v+u 
		& 
		{\rm in} \ \Omega \times (0,T), 
		\\[2mm] 
		\nabla u \cdot \nu 
		= 
		\nabla v \cdot \nu 
		= 0 
		& {\rm on} \ \partial\Omega \times (0,T), 
		\\[1mm] 
		u(\cdot, 0)=u_0, 
		\ v 
		(\cdot, 0)=v_0 
		& {\rm in} \ \Omega, 
	\end{cases}
\end{equation}
where $\Omega$ is a general bounded domain 
in $\Rn$ ($n\in \mathbb{N}$)
with smooth boundary $\pa \Omega$ 
and $\nu$ 
is the outward normal vector 
to $\pa\Omega$ and $T>0$. 
The initial data 
$u_0$ and $v_0$ 
are 
supposed to be nonnegative functions 
such that $u_0 \in C(\overline{\Omega})$ 
and 
$v_0 \in 
C^1(\overline{\Omega})$. 
Also we assume that 
\begin{align*}
\alpha>0, \quad \chi>0, \quad m_1, m_2
\in \mathbb{R}.
\end{align*}
In the system \eqref{KS}, 
the unknown function	 
$u(x,t)$ represents the density of the cell 
population and the unknown function $v(x,t)$ 
shows the concentration of the signal substance 
at place $x$ and time $t$.
The 
system \eqref{KS} with the 
simplest choices 
$m_1=1$ and $m_2=2$ 
describes a part of life cycle of 
cellular slime molds with chemotaxis and it 
was proposed by 
Keller--Segel \cite{Keller-Segel_1970} 
in 1970. 
After that, 
a quasilinear system such as \eqref{KS} 
was proposed by Painter--Hillen 
\cite{Painter-Hillen_2002}. 
A number of variations of the original 
Keller--Segel system 
are proposed and investigated 
(see e.g., 
Bellomo--Bellouquid--Tao--Winkler 
\cite{Bellomo-Bellouquid-Tao-Winkler_2015}, 
Hillen--Painter \cite{Hillen-Painter_2009} 
and 
Horstmann 
\cite{Horstmann_2003,Horstmann_2004}). 
%

\indent
According to a continuity 
model, the first equation in \eqref{KS}
has the flux vector 
$F=-\left[(u+\alpha)^{m_1-1}\nabla u-\chi 
u(u+\alpha)^{m_2-2}\nabla v\right]$. 
We can recognize that 
$(u+\alpha)^{m_1-1}\nabla u$ represents 
the diffusive flux and 
$-\chi u(u+\alpha)^{m_2-2}\nabla v$ represents 
the chemotactic flux 
modelling 
undirected cell migration 
and the advective flux with velocity 
dependent 
on the gradient of the signal. 
More precisely, when cellular slime molds 
plunge into hunger, they move towards higher 
concentrations of the chemical substance 
secreted by cells.

\indent
From a mathematical point of view, 
$u$ in 
\eqref{KS}
enjoys the mass conservation 
property: 
\begin{align*}
	\int_{\Omega}u(\cdot,t)=\int_{\Omega}u_0
\end{align*}
for all $t \in (0,T)$. 
It is a meaningful question 
whether solutions of \eqref{KS} 
remain bounded 
or blow up. 
As to this question, 
it is known that 
the borderline 
between boundedness and 
blow-up is 
the case that $m_2=m_1+\frac{2}{n}$, 
$m_1 \ge 1$, $m_2 \ge 2$. 
According to the result 
established by 
Horstmann--Winkler 
\cite[Theorems 4.1 and 6.1]
{Horstmann-Winkler_2005} 
in the case $m_1=1$, 
it can be expected that 
\eqref{KS} has a global bounded solution 
in the case that 
$m_2<m_1+\frac{2}{n}$ 
and 
a blow-up solution 
in the case that 
$m_2>m_1+\frac{2}{n}$.  
Indeed, in the case that 
$\Omega$ is a bounded domain and 
$m_2<m_1+\frac{2}{n}$, 
there exists 
a global bounded solution 
of 
\eqref{KS} (see Tao--Winkler 
\cite{Tao-Winkler_2012}, 
Ishida--Seki--Yokota
\cite{Ishida-Seki-Yokota_2014} and 
Senba--Suzuki 
\cite{Senba-Suzuki_2006}). 
%
In addition, this result 
was shown also for the 
degenerate chemotaxis system 
(\eqref{KS} 
with $\alpha =0$) 
(see Ishida--Yokota 
\cite{Ishida-Yokota_2012-2, 
Ishida-Yokota_2020} 
when $\Omega = \Rn$ and  
$m_2 < m_1 + \frac{2}{n}$;  
\cite{Ishida-Seki-Yokota_2014} when $\Omega$ 
is a 
bounded domain 
and  
$m_2 < m_1 + \frac{2}{n}$; 
Mimura 
\cite{Mimura_2017} 
when $\Omega$ is a bounded domain with 
Dirichlet--Neumann boundary condition, 
$m_2=2$ and 
$m_1 > 2-\frac{2}{n}$). 
If $m_2 \ge m_1+\frac{2}{n}$, then 
the results are  
divided by the size of initial data. 
For example, 
the system \eqref{KS} has a 
global solution 
with small initial data when $\Omega 
= \Rn$ and $\alpha=0$ 
even if 
$m_2 \ge m_1+\frac{2}{n}$ 
(see Ishida--Yokota 
\cite{Ishida-Yokota_2012}). 
On the other hand, 
in the case that 
$\Omega = B_R 
:= \{x\in \Rn \big| |x| <R \}$ 
$(R>0)$, 
$m_1=1$, $m_2=2$, 
$n \ge 3$, which implies 
$m_2> m_1 + \frac{2}{n}$, 
there exist 
initial data such that the 
radially symmetric solution 
of \eqref{KS} blows up in finite time 
(see Winkler \cite{Winkler_2013}). 
The result was extended to the case that 
$\Omega = B_R$, 
$m_2>m_1+\frac{2}{n}$, $n \ge 2$ 
(see Cie\'{s}lak--Stinner 
\cite{Cieslak-Stinner_2012, 
Cieslak-Stinner_2014} when $\alpha>0$, 
Hashira--Ishida--Yokota 
\cite{Hashira-Ishida-Yokota_2018} when 
$\alpha=0$). 
In the most important case that 
$\Omega = B_R$, $m_1=1$, $m_2=2$, $n=2$, 
which implies $m_2=m_1+\frac{2}{n}$,
there 
exist initial data such that the corresponding   
solutions of \eqref{KS} 
blow up 
in finite time 
(see Mizoguchi--Winkler 
\cite{Mizoguchi-Winkler}).

\indent
We are especially interested in 
a {\it lower bound\/ }%
for the blow-up time 
for 
solutions of \eqref{KS},  
because  
it seems to be 
important to know how 
$m_1$ 
affects on the blow-up time 
for solutions of \eqref{KS}. 
The study 
of a lower bound for 
the blow-up time seems 
to be interesting 
widely for general 
parabolic systems 
(see Payne--Schaefer 
\cite{Payne-Schaefer_2006} and 
Enache 
\cite{Enache_2011}), 
wave equations 
(see Philippin 
\cite{Philippin_2015}) 
and heat equations 
(see Payne--Philippin--Vernier Piro 
\cite{Payne-Philippin-Piro_2010}). 
Moreover, 
explicit lower bounds for the blow-up time 
for solutions of 
various semilinear parabolic equations 
were obtained by \cite{Payne-Schaefer_2006}. 
As to chemotaxis systems, 
Payne--Song 
\cite{Payne-Song_2010, Payne-Song_2012} 
established a lower bound of blow-up time 
for solutions of \eqref{KS} with 
$m_1=1$ and $m_2=2$ 
in the form 
\begin{align*}
\widetilde{t}^{\ast} 
\ge \int_{\Phi_1(0)}^{\infty}
\frac{d\xi}{V{\xi}^{\frac{3}{2}}+
W{\xi}^2} \quad (n=2)  
\end{align*}
and 
\begin{align*}
\widetilde{t}^{\ast} \ge \int_{\Phi_1(0)}
^{\infty}
\frac{d\xi}{X{\xi}^{\frac{3}{2}}+
Y{\xi}^3} \quad (n=3); 
\end{align*}
note that $\widetilde{t}^{\ast}$ means 
the blow-up time  
in $\Phi_1$-measure, 
i.e., 
$\lim_{t \nearrow \widetilde{t}^{\ast}}
\Phi_1(t)
= \infty$, where $\Phi_1(t)$ 
is defined as 
\begin{align}\label{oldphi}
\Phi_1(t)
:=\kappa \int_{\Omega}{u(\cdot,t)}^2 
+ \int_{\Omega} |\Delta v(\cdot,t)|^2 \quad 
(t>0)
\end{align}
with some $\kappa>0$. 
When $\Omega$ is a convex bounded domain and 
$m_1=1$, 
Li--Zheng \cite{Li-Zheng_2013} gave  
a lower bound for the blow-up time for 
solutions of \eqref{KS} by using \eqref{oldphi} 
in the case that 
$m_2 \in (\frac{5}{3},2]$, 
$n=3$ and in the case that 
$m_2 \in [2,3)$, 
$n=2$. 
After that, 
when $\Omega=B_1$, $\alpha=1$ 
and $m_1=1$, 
in the case that 
$m_2 \in [\frac{5}{3},3]$ and $n=3$, 
Tao--Vernier Piro 
\cite{Tao-Piro_2016} introduced 
the measure $\Phi_2(t)$ 
in the form 
\begin{align}\label{newphi}
	\Phi_2(t)
:=
		\int_{\Omega} \left(u(\cdot,t) 
		+ 
		1\right)^p 
		+
		\int_{\Omega} \left|\nabla v(\cdot,t)
\right|^{2q} \quad 
(t>0)
\end{align}
for suitable $p,q>1$ 
($p=2$ and $q=4$ when 
$m_2 \in [\frac{5}{3},2]$;  
$p=5$ and $q=22$ when  
$m_2 \in (2,3]$) 
from the view point of local existence 
of classical solutions to \eqref{KS} and 
an initial datum $v_0 \in W^{1,q}(\Omega)
$ 
(see \cite[Lemma 3.1]
{Bellomo-Bellouquid-Tao-Winkler_2015}). 
This restriction on 
$m_2$ and $n$ was removed by 
Anderson--Deng \cite{Anderson-Deng_2017} 
when $\Omega$ is a 
convex bounded domain 
and 
$m_1=1$. 
Furthermore, 
as a new attempt to 
estimating a lower bound 
for the blow-up time in the above sense, 
Marras--Vernier Piro--Viglialoro 
\cite{Marras-Piro-Viglialoro_2015, 
Marras-Piro-Viglialoro_2016} 
obtained a lower bound for the blow-up 
time of the 
more generalized equation with a source 
term: 
\begin{equation}\label{source}
	\begin{cases}
		u_t=\nabla \cdot [
		\nabla u-k_1(t) u^{m_2-1}
		\nabla v] + {\boldmath f(u)}
		& {\rm in} \; \Omega \times (0,T),
		\\[1mm]
		v_t=k_2(t)\Delta v-k_3(t)v+k_4(t)u
		& {\rm in} \; \Omega \times (0,T)
		\\[1mm]
	\end{cases}
\end{equation}
under Neumann boundary conditions 
and initial conditions, 
where $k_i(t)\ (i=1,2,3,4)$ are nonnegative 
smooth functions of $t$,  
$m_2 \in [2,3)$ when $n=2$, 
$m_2 \in (\frac{5}{3},2)$ when $n=3$, 
$f$ satisfies $f(u) \le cu^2$ with $c>0$. 
A similar result 
for the parabolic--elliptic version 
of \eqref{source}
was deduced by Jiao--Zeng 
\cite{Jiao-Zeng_2018}. 

\indent 
Now we focus on the 
studies 
obtained by 
\cite{Tao-Piro_2016} and 
\cite{Anderson-Deng_2017} 
which 
gave 
a lower bound for 
the blow-up time for solutions of \eqref{KS} 
under the following conditions:
\begin{itemize}
\item 
``$m_1=1$", 
$m_2 \in [\frac{5}{3},3]$, 
$n=3$, 
$\Omega$ is a 
{\it unit ball }%
$B_1 \subset \boldmath 
\mathbb{R}^3$\ (\cite{Tao-Piro_2016}); 
\item 
{``$m_1=1$"}, 
$m_2\in \mathbb{R}$, 
$n \in \mathbb{N}$, 
$\Omega$ is a 
{\it convex }%
bounded domain 
in $\Rn$\ (\cite{Anderson-Deng_2017}).
\end{itemize}
However, there is 
still room
for improvements 
in these results. 
More precisely, 
we cannot find any results in 
the nonlinear case that $m_1 \neq 1$ and $
\Omega$ is 
a non-convex bounded domain.  
Hence 
the current  situation is summarized in 
Table \ref{sumarry}.
\vspace{2mm}
\begin{table}[h]
\centering
{\renewcommand\arraystretch{1}
\begin{tabular}{|l||c|c|c|}
\hline
\backslashbox[31mm]
  &
	$\Omega$:\ ball &
	$\Omega$:\ convex &
	$\Omega$:\ non-convex
\\\hhline{|=#=|=|=|}
	\begin{tabular}{c}
Linear case \\[-1mm] ($m_1=1$)
\vspace{1mm}
\end{tabular}    &
  Tao--Vernier Piro \cite{Tao-Piro_2016} &
  Anderson--Deng \cite{Anderson-Deng_2017} &
	{\bf No work}
\\\hline
 \begin{tabular}{c}
Nonlinear case \\[-1mm] \hspace{-6mm}($m_1 \neq 
1$)\vspace{1mm}
\end{tabular} &
  {\bf No work} &
  {\bf No work} &
 	{\bf No work} 
\\\hline
\end{tabular}
}
\caption{The known results 
on lower bounds for the blow-up time 
in \eqref{KS}}\label{sumarry}
\end{table}

\noindent
Here,  
if some results 
can be given in 
the nonlinear case that $m_1 \neq 1$, 
then the following natural question 
arises: 
%
\begin{align*}\tag*{{\bf (Question)}}\label{Q}
\text{\it How does $m_1$ 
affect on the blow-up time 
for solutions of \eqref{KS}? }
\end{align*} 
Since the 
blow-up for solutions of the system \eqref{KS} 
describes the aggregation 
of cells and 
strong diffusion seems to 
prevent the aggregation 
and to cause delay in the blow-up,  
we can intuitively 
conjecture 
the answer to this question as follows: 
\begin{align*}\tag*{{\bf (Conjecture)}} 
\label{C}
\text{\it The larger $m_1$ is, the larger 
the blow time $t^{\ast}$ 
for solutions of \eqref{KS} is.}
\end{align*}

\indent
The first purpose of this paper 
is completely 
to fill in {\bf ``No work''} 
in Table \ref{sumarry}. 
The second purpose of this paper is
to present 
an answer to the above question 
and justify the above conjecture, 
that is, to give an explanation for 
effect of 
nonlinear diffusion 
and the 
chemotaxis term for the 
blow-up time 
in a parabolic--parabolic chemotaxis system. 

\indent
Furthermore, we should mention 
how we can derive an explicit 
lower bound 
for the blow-up time for solutions 
of \eqref{KS}. In the previous works, 
the blow-up time for classical solutions of 
\eqref{KS} can be obtained by using the
energy function $\Phi_2(t)$ 
defined as \eqref{newphi}. 
However, there is a gap 
between the blow-up time for 
$\Phi_2(t)$ with $L^p(\Omega) \times 
W^{1,2q}(\Omega)$-norm 
of $(u,v)$ and that for solutions 
in the classical sense, i.e., 
in the sense of $L^{\infty}(\Omega)$-norm 
of $u$ (for details see 
Definition \ref{tdf}). 
Indeed, 
assume that $\Omega$ 
is a bounded domain. Then we know that 
\begin{align}\label{norm}
\|u(\cdot,t)\|_{L^{p}(\Omega)} \le 
|\Omega|^{\frac{1}{p}}\|
u(\cdot,t)\|
_{L^{\infty}(\Omega)}.
\end{align}
In view of \eqref{norm} 
we see that if a solution of 
\eqref{KS} blows up 
in $L^p(\Omega)$-norm, then it 
blows up in $L^\infty(\Omega)$-norm 
at the same time; however,  
even if 
a solution of \eqref{KS} 
blows up in $L^{\infty}(\Omega)$-norm, we 
cannot predicate whether the solution 
blows up or not 
in $L^p(\Omega)$-norm. 
From the numerical resolution method, 
it seems that the 
blow-up time for solutions 
of \eqref{KS} 
in $\Phi_2$-measure has a long delay 
(see 
Farina--Marras--Viglialoro 
\cite
[FIGURE 1]
{Ferina-Marras-Vigilialoro_2015}).

Another purpose of this paper is to 
bridge a gap between
the blow-up time for 
solutions of \eqref{KS} 
in $\Phi_2$-measure 
and that 
in the classical sense. 
The key to accomplishing this purpose is 
a refined extensibility criterion 
established 
by Freitag 
\cite[Theorem 2.2]{Freitag_2018}.

\indent
Before stating
the main result, 
we define classical solutions of \eqref{KS} 
and the 
blow-up time as follows: 
\begin{df}
A pair $(u,v)$ is 
called a classical solution 
of \eqref{KS} if 
\begin{align*}
	&
	u \in C\left(\overline{\Omega} 
	\times [0,T)\right)
	\cap 
	C^{2,1}\left(\overline{\Omega} 
	\times (0,T)\right),\\[0mm]
	&v \in C\left(\overline{\Omega} 
	\times [0,T)\right)
	\cap 
	C^{2,1}\left(\overline{\Omega} 
	\times (0,T)\right)
	\cap
	L_{\rm loc}^{\infty}
	\left([0,T);W^{1,\infty}(\Omega)
\right)
\end{align*}
and $u,v$  
satisfy \eqref{KS} 
	in the classical sense.
\end{df}
\begin{remark}
Local existence 
and uniqueness 
of classical solutions to 
\eqref{KS} are  
known (see Lemma \ref{or} 
and Remark \ref{local} (i) below). 
\end{remark}
\begin{df}\label{tdf}
	Let $t^{\ast}$ be a 
	maximal 
	time 
	for which a solution of \eqref{KS} 
exists 
	for $0 \leq t < t^{\ast}$. 
	Then $t^{\ast}$ 
	is called a {\it blow-up time in the 
classical sense} if 
	$t^{\ast} < \infty$ and 
\begin{align}\label{btc}
	\lim\limits_{t \nearrow t^{\ast}} 
	(\|u(\cdot,t)\|_{L^{\infty}(\Omega)}
	+
	\|v(\cdot,t)
	\|_{W^{1,\infty}(\Omega)})
	=\infty. 
\end{align}
\end{df}
In order to state the main theorem 
we shall give the following conditions 
for the parameters $p>1$ and $q>\frac{1}{
\eta-1}$:
	\begin{align}\tag{C1}\label{C1}
		&
		p> 
		\max 
		\left\{ 
		\frac{n}{2}(m_2-m_1),\ 
		n(m_2-m_1-1),\ 
		n
		\right\},
		\\[3mm]\tag{C2}\label{C2} 
		& 
		p>{\rm max} 
		\left\{ 
		\frac{q(2m_2-m_1-3)}{q\eta-q-1},\ 
		-2m_2+m_1+3,\ 
		\frac{2q}{q\eta-q+1},\ 
		\frac{\eta(m_1-1)}{(\eta-1)
		(\eta-2)}
		\right\},
	\end{align}
	where $\eta$ is defined as 
	\begin{align}\label{df of eta}
\begin{cases}
\eta \in (1,2)\ \text{is any} & (n=1,2), 
\\[2mm]
\eta := \dfrac{n}{n-1} & (n \ge 3). 
\end{cases}
\end{align}

\indent
 We now state the main result of this paper. 
The main result 
gives a lower bound 
for the blow-up time 
for solutions of \eqref{KS} 
with nonlinear diffusion. 
\begin{thm}\label{main}
Let $t^{\ast}<\infty$ be 
the blow-up time in the classical sense 
for a classical solution $(u,v)$ 
of \eqref{KS}. 
Then 
there exist constants 
$A=A(m_1)>0$, $B>0$, $C=C(m_1)>0$, 
$D \ge 0$ and $p>1$, $q>\frac{1}{\eta-1}$
fulfilling \eqref{C1}, \eqref{C2} such that 
\begin{align}\label{blow-up-time}
		t^{\ast} \geq 
		\int_{\Phi(0)}^{\infty} 
		\frac{d\tau} 
		{A{\tau}^{f(\eta,r)} 
		+ 
		B{\tau}^{f(\eta,1)} 
		+ 
		C\tau^{\eta}%
		 +
		D},
	\end{align}
	where $\Phi > 0$, $f>1$ and $r>0$ 
	are defined as 
	\begin{align}\label{def of phi}
		&
		\Phi(t) := \frac{1}{p} 
		\int_{\Omega} \left(u(\cdot,t) 
		+ 
		\alpha\right)^p 
		+ \frac{1}{q}
		\int_{\Omega} |\nabla v(\cdot,t)|^{2q} 
\quad (t>0), 
		\\[2mm]\label{def of f}
		&f(\eta,s) := 1+ 
\frac{\eta-1} 
		{n\left
(\frac{1}{n}-\frac{\eta}{2}
+\frac{1}
		{2s}\right)} 
\quad (s>0), 
		\\[1mm]\label{def of r}
		&
r = r(m_1):= \frac{p}{p+m_1-1}, 
	\end{align}
and $\eta \in (1,2)$ 
is defined as \eqref{df of eta}. 
\end{thm}
\begin{remark}
	Theorem \ref{main} covers the case that 
	$\Omega$ is a 
	general non-convex bounded domain and 
$m_1 = 1$ 
(see \cite[Theorem 1.1]
	{Anderson-Deng_2017}). 
	Moreover, the constants
	$A,B,C,D$ in \eqref{blow-up-time} 
are better than the previous. 
	If $\Omega$ is a 
convex bounded domain, then 
	we can take $D=0$ 
	(see Corollary \ref{vremark} below). 
\end{remark}
\begin{remark}
The estimate \eqref{blow-up-time} implies 
that the larger $m_1$ is, the larger 
the blow-up time $t^{\ast}$ 
for solutions $(u,v)$ of \eqref{KS} is. 
Indeed, we shall 
consider the effect of $m_1$. 
When we fix $p,q>1$, $f(\eta,r(m_1))$ is 
decreasing in $m_1$. 
This entails that if $m_1$ is 
	sufficiently large, then 
	$t^{\ast}$ will be large. 
In other words, 
if the power of diffusion is strong, 
then a solution of \eqref{KS} 
blows up with a delay. 
	On the other hand, we automatically see 
	that if the power 
	of diffusion is weak, then 
	a solution of \eqref{KS} 
	blows up early.
\end{remark}
\begin{remark}
As to the assumption of Theorem \ref{main}, 
we need not suppose that 
a solution of \eqref{KS} blows up 
in finite time in the proof 
of Theorem \ref{main}. Namely, 
this means that 
we can essentially estimate 
the ``life span" for solutions 
of \eqref{KS}. Since we are interested in 
the blow-up time  for solutions of \eqref{KS}, 
the assumption concerning 
the blow-up
is added 
in Theorem \ref{main}.  
%
\end{remark}
By computing the integral appearing 
in \eqref{blow-up-time}, 
we can establish 
a lower bound for $t^{\ast}$ 
in the simple form as follows: 
\begin{corollary}
Under the assumption of 
Theorem \ref{main}, 
if\/ $\Phi(0)<1$, then \eqref{blow-up-time} 
is 
rewritten as follows\/{\rm :} 
	\begin{align*}
		& s \geq 1 
\ \, \Longrightarrow 
\ \, t^{\ast} 
		\geq \frac{1}
		{f(\eta,r)-1}\cdot 
\frac{\Phi(0)}{A{\Phi(0)}^{f(\eta,r)-1}+
		B{\Phi(0)}^{f(\eta,1)-1}
+C{\Phi(0)}^{\eta-1}
		+D},
		\\[1mm]
		& s < 1 
\ \, \Longrightarrow 
\ \, t^{\ast} 
		\geq \frac{1}
		{f(\eta,1)-1}
\cdot \frac {\Phi(0)}{
		B{\Phi(0)}^{f(\eta,1)-1}+
		A{\Phi(0)}^{f(\eta,r)-1}
+C{\Phi(0)}^{\eta-1}
		+D}.
	\end{align*}
\end{corollary}

\medskip
The strategy for the proof of Theorem 
\ref{main} is to 
derive an ordinary differential inequality 
for $\Phi(t)$ defined as \eqref{def of phi}. 
%
%
%
%
%
We first 
construct 
the inequality 
	\begin{align}\nonumber
		&\frac{d\Phi}{dt}
		+\frac{p-1}{2}{\left(\frac{2}{p+m_1-1}
		\right)}^2\int_{\Omega}\left|\nabla 
		(u+\alpha)^
		{\frac{p+m_1-1}{2}}\right|^2
		+\left(\frac{2(q-1)}{q^2}-\delta 
		\right)\int_{\Omega}
\left|\nabla |\nabla v|^q\right|^2
		\\\nonumber
		&\leq
		\frac{{\chi}^2(p-1)}{2}\int_{\Omega} 
		(u+\alpha)^
		{p+2m_2-m_1-3}|\nabla v|^2
		+\frac{4(q-1)+n}{2}
		\int_{\Omega} (u+\alpha)^2 
		|\nabla v|^
		{2q-2}
		+D_{\delta}
	\end{align} 
for some 
$\delta \in (0,\frac{2(q-1)}{q^2})$ and 
$D_\delta>0$. We next estimate 
the first and second terms 
on the right-hand side 
%
by using Young's inequality and 
H\"older's inequality 
to make 
$\int_{\Omega} 
\left(u+\alpha\right)^{p\eta}$. 
%
%
In \cite{Anderson-Deng_2017} dealing 
with the case that $m_1=1$, 
by applying 
the Sobolev embedding 
$W^{1,1}(\Omega) \hookrightarrow 
L^\eta(\Omega)$, 
the quantity 
$\int_{\Omega} 
\left(u+\alpha\right)^{p\eta}$ 
is estimated 
as 
\begin{align*}
\int_{\Omega} 
\left(u+\alpha\right)^{p\eta}
\leq
C 
\left(
\int_{\Omega}(u+\alpha)^{p}
+
\int_{\Omega}
|\nabla(u+\alpha)^{p}|
\right)^\eta
\end{align*}
with some $C>0$, 
and hence we need an 
additional deformation to obtain 
$\int_{\Omega} | \nabla (u+\alpha)^
\frac{p}{2} |^2$, 
because of the difference 
from 
$\int_{\Omega} |\nabla(u+\alpha)^{p}|^2$. 
Our technical innovation 
in this paper 
is to apply 
the Gagliardo--Nirenberg inequality 
instead of using the Sobolev embedding 
as 
\begin{align*}
&\int_{\Omega} 
\left(u+\alpha\right)^{p\eta}
\leq c 
\left\|\nabla (u+\alpha)^
{\frac{p}{2}}
\right\|_
{L^2(\Omega)}^
{\eta}
\left\| (u+\alpha)^
{\frac{p}{2}}
\right\|
_{L^
{2}
(\Omega)}^
{\eta}
+\widetilde{c}
\left\| (u+
\alpha)^
{\frac{p}{2}}
\right\|
_{L^
{2}
(\Omega)}^
{2\eta} 
\end{align*}
for some $c,\widetilde{c}>0$ 
in the case $m_1 = 1$ 
(we extend 
this inequality 
to the case $m_1 \neq 1$). 
We thus obtain 
the factor 
$\int_{\Omega}| \nabla (u+\alpha)^
\frac{p}{2} |^2$ 
directly and 
a sharp lower bound 
for the blow-up time can be 
established. 
In addition, 
the key to removing 
the convexity of $\Omega$ 
is the estimate for 
$\int_{\partial\Omega}|\nabla v|
		^{2q-2}\nabla\left(|\nabla v|^2\right)
		\cdot\nu$ 
which is estimated by $0$ 
in the previous works 
\cite{Anderson-Deng_2017} 
and 
\cite{Tao-Piro_2016}. 
In this paper 
it is estimated by the 
combination of the embedding 
$W^{\beta+\frac{1}{2}}(\Omega) 
\hookrightarrow 
L^2(\partial \Omega)\ 
(\beta \in (0,\frac{1}{2}))
$ 
with the fractional 
Gagliardo--Nirenberg inequality. 

\indent
This paper is organized as follows.
In Section \ref{Section 2} 
we will collect lemmas which 
will be used in this paper. 
In Section \ref{Section 3} we will 
present an estimate for 
the first term of $\Phi(t)$ 
defined in Theorem \ref{main}. 
In Section \ref{Section 4} 
we will give 
an estimate for 
the second term of $\Phi(t)$. 
We will complete 
the proof of Theorem \ref{main} 
in 
Section \ref{Section 5} 
through a series of four steps. 
An important thing is 
to obtain  
an ordinary differential inequality 
of $\Phi(t)$ without wasting 
effect of $m_1$.  
\section{Preliminaries}\label{Section 2}
In this section 
we recall some known 
basic results. 
%
Let us begin with 
the well-known 
Gagliardo--Nirenberg inequality 
(for details, see e.g.,  
Li--Lankeit 
\cite[Lemma 2.3]{Li-Lankeit_2016}): 
\begin{lem}\label{G-N}
Suppose that $\Omega$ 
is a bounded domain in\/ $\Rn$ with 
smooth boundary. 
	Let 
	$r\geq1$, $0<q \leq p\leq\infty$, $s>0$ 
	be such that 
	$\frac{1}{r}\leq \frac{1}{n}+\frac{1}{p}$.
	Then there exists 
	$c>0$ such that 
	\begin{align*}
		\|w\|_{L^p(\Omega)}\leq c
		\left(\|\nabla w\|_{L^r(\Omega)}
		^{a}\|w\|_{L^q(\Omega)}^{1-a}
	+\|w\|_{L^s(\Omega)}\right)
	\end{align*}
for all $w\in W^{1,p}(\Omega)\cap 
	L^q(\Omega)$,
	where 
$a:=\frac{\frac{1}{q}-\frac{1}{p}}
		{\frac{1}{q}+\frac{1}{n}-
		\frac{1}{r}}$. 	
\end{lem} 
Next we give an 
estimate for a particular boundary 
integral which 
enables us to 
cover possibly non-convex bounded 
domains 
(see \cite[Lemma 2.1]{Li-Lankeit_2016}). 
\begin{lem}\label{non-convex}
	Let $\Omega$ be 
a bounded domain in\/ $\mathbb{R}^{n}$ 
with smooth boundary. 
	Suppose that $q\in[1,\infty)$. 
	Then for all 
$\delta>0$ there exists 
	$C_{\delta}>0$ independent of $q$ such that 
for all 
$w\in C^2(\overline{\Omega})$ satisfying 
$\frac{\partial w}{\partial \nu}=0 
\ {\rm on}\ \partial\Omega$, 
	\begin{align*}
		\int_{\partial\Omega} |\nabla w|
		^{2q-2}\frac{\partial|\nabla w|^2}
		{\partial \nu}\leq \delta 
		\int_{\Omega} 
\left|\nabla|\nabla w|^q\right|
		^2
+C_{\delta}. 
	\end{align*}
	\end{lem}
If 
$\Omega$ is a 
convex bounded 
domain, then the following holds 
(see \cite[Lemma 3.2]
{Tao-Winkler_2012}):
\begin{lem}\label{convex}
Assume that $\Omega$ is a convex bounded 
domain, 
	and that $w\in C^2(
	\overline{\Omega})$ 
	satisfies $\frac{\partial w}{\partial
	\nu}=0\ {\rm on}\ 
	\partial\Omega$. Then
	\begin{align*}
		\frac{\partial |\nabla w|^2}{\partial
		\nu}\leq 0 \quad 
		{\rm on}\ \partial\Omega.
	\end{align*}
\end{lem}
%
%
%
%
%
%
We finally introduce 
the fundamental fact 
for classical solutions of \eqref{KS} and 
results  
for the blow-up time. 
We recall the result 
for local existence of 
classical solutions 
(see \cite[Lemma 1.1]{Tao-Winkler_2012}). 
\begin{lem}\label{or}
Let $u_0 \in C(\overline{\Omega})$ and 
$v_0 \in C^1(\overline{\Omega})$. 
Then there 
exist  
$T_{\rm max} \in (0,\infty]$ 
and a uniquely determined 
pair $(u,v)$ of nonnegative functions 
in $C(\overline{\Omega}\times [0,T_{\rm max}))
\cap C^{2,1}(\overline{\Omega} \times 
(0,T_{\rm max}))$ solving $\eqref{KS}$ 
classically in $\Omega \times (0,T_{\rm max})$. 
Additionally we either have 
\begin{align*}
T_{\rm max} = \infty \quad \text{or}%
\quad \limsup_{t 
\nearrow
T_{\rm max}}\,  
\left(\|u(\cdot,t)\|_{L^{\infty}(\Omega)} 
+
\|v(\cdot,t)\|_{L^{\infty}(\Omega)}
\right) 
=\infty. 
\end{align*}
\end{lem} 
The following lemma,  
which was proved in \cite[Lemma 4.1]
{Horstmann-Winkler_2005},  plays 
an important role in considering  
the blow-up time for solutions of 
\eqref{KS} defined in Definition 
\ref{tdf}. %
\begin{lem}\label{W}
Let 
$(u,v)$ be a classical solution of 
\eqref{KS}. 
Suppose that there 
exist 
$p \ge 1$ and $C>0$ 
such that  
\begin{align*}
\| u(\cdot,t )\|_{L^p(\Omega)} 
\le C \quad \text{for all} 
\ t \in (0,T). 
\end{align*}
Then 
\begin{align*}
\|v\|_{L^{\infty}\left(0,T;W^{1,q}(\Omega)
\right)}
<\infty
\end{align*}
for any $q \in [1,\frac{np}{(n-p)_{+}})$ and 
even $q=\infty$ if $p>n$. 
\end{lem}
\begin{remark}\label{not-v-but-u}
Let a pair of $(u,v)$ 
solve \eqref{KS} classically. 
As to 
$\Phi(t)$ 
defined as 
\eqref{def of phi}, 
we note that 
it is sufficient only 
to deal with the blow-up time for 
$u$ in $L^p(\Omega)$-norm 
under the condition $p>n$ guaranteed by 
\eqref{C1}. 
In other words, 
the blow-up time for 
$v$ does not affect 
on that 
for $\Phi(t)$. 
We should 
explain that 
the blow-up time for $u$ 
in $L^p(\Omega)$-norm is 
larger than or equal to that for $v$ in 
$W^{1,\infty}(\Omega)$-norm 
(see Definition \ref{tdf}). 
Indeed, 
by the contraposition of Lemma \ref{W}, 
we can find that 
if 
$v$ blows up in 
$W^{1,\infty}(\Omega)$-norm, 
then 
$u$ blows up in 
$L^{p}(\Omega)$-norm 
for all $p > n$, 
and hence 
the blow-up time for $u$ 
in $L^p(\Omega)$-norm is 
larger than or equal to that for $v$ in 
$W^{1,\infty}(\Omega)$-norm 
under the condition $p>n$. 
\end{remark} 
\begin{remark}\label{local}
The condition \eqref{btc} in the 
definition of 
a blow-up time in the classical sense 
can be replaced with 
\begin{align}\label{btca}
\lim\limits_{t \nearrow t^{\ast}} 
	(\|u(\cdot,t)\|_{L^{\infty}(\Omega)}
	+
	\|v(\cdot,t)
	\|_{W^{1,k}(\Omega)})
	=\infty, 
\end{align}
where $k>n$ because of the condition 
$q>\frac{1}{\eta-1}=n-1$ 
in Theorem \ref{main}. 
Indeed, 
we see from Remark 
\ref{not-v-but-u} and 
the continuous embedding 
$W^{1,\infty}(\Omega) 
\hookrightarrow W^{1,2q}(\Omega)$ 
for all $q \ge 1$ and  
$W^{1,2q}(\Omega) \hookrightarrow 
L^{\infty}(\Omega)$ 
for $q>\frac{n}{2}$
that if 
$v$ in $W^{1,2q}(\Omega)$-norm 
or in $L^{\infty}(\Omega)$-norm blows up, 
then that in $W^{1,\infty}(\Omega)$-norm 
also blows up 
for $n \ge 2$. 
An argument 
similar to that in 
Remark \ref{not-v-but-u} 
implies 
that \eqref{btc} can be replaced with 
\eqref{btca}. 
\end{remark} 
The following lemma enables us to 
show that 
a maximal existence time 
results in unboundedness in 
$L^p$-spaces for smaller $p \in [1,\infty)$ 
(see \cite[Theorem 2.2]{Freitag_2018}). 
\begin{lem}\label{smaller-p}
Let $u_0 \in C(\overline{\Omega})$ and $v_0 
\in C^1(\overline{\Omega})$. 
If a 
solution $(u,v)$ of \eqref{KS} 
in $\Omega \times (0, T_{\rm max})$ 
has a blow-up time $T_{\rm max}<\infty$ 
in the classical sense, 
then there exists $p \ge 1$ fulfilling 
\[
p \ge 
		\max 
		\left\{ 
		\frac{n}{2}\left(m_2-m_1\right),\ 
		n\left(m_2-m_1-1\right) 
		\right\}
\]
such that 
\begin{align*}
\limsup_{t \nearrow T_{\rm max}}\, 
\left\|u(\cdot,t)\right\|_{L^{p_0}(\Omega)} 
= \infty 
\end{align*}
for all 
$p_0 > p$.
\end{lem}
In the proof of Theorem \ref{main} 
we will use 
the following corollary 
in order to remove a gap 
between the blow up time for solutions 
$(u,v)$ of \eqref{KS} in the classical sense 
and that in $\Phi$-measure. 
\begin{corollary}\label{non-gap} 
Let $1\le p,q \le \infty$. 
Let $t^{\ast}$ be the blow-up time in the 
classical sense 
and $t_{p,q}^{\ast}$ the blow-up time 
in the measure $\Phi(t)$ defined as 
\eqref{def of phi}\/{\rm :} 
\begin{align*}
\lim
\limits_{t \nearrow t_{p,q}
^{\ast}} \Phi(t) = \infty. 
\end{align*}
Then under the condition \eqref{C1}, 
\begin{align}\label{time-equal}
t^{\ast} = t_{p,q}^{\ast}.
\end{align}
\end{corollary}
\begin{proof}
We obtain from the continuous embedding 
$L^\infty(\Omega) \hookrightarrow 
L^p(\Omega)$ 
such as  \eqref{norm} that 
if $\Phi(t)$ blows up at $t=t_{p,q}^\ast$, 
then 
$u$ in $L^{\infty}(\Omega)$-norm  
also blows up, 
and hence it is obvious that 
\begin{align*}
t^{\ast} \le 
t_{p,q}^{\ast}. 
\end{align*} 
Here we note from 
Remark \ref{not-v-but-u} that 
the blow-up for $\Phi(t)$ 
implies that for $u$ in $L^{p}(\Omega)$-norm. 
On the other hand, 
we can find from 
Lemma \ref{smaller-p} that 
the blow-up in the classical sense 
implies that in $\Phi$-measure 
for $p$ satisfying \eqref{C1}.  
Therefore, under the condition 
\eqref{C1}, 
we can attain that
\begin{align*}
t^{\ast} \ge 
t_{p,q}^{\ast}.
\end{align*} 
Thus we obtain \eqref{time-equal}. 
\end{proof}
Hereafter, 
we assume that a pair 
$(u,v)$ is a classical solution of 
\eqref{KS}.
\section
{An estimate for 
$\frac{1}{p}\int_{\Omega}
\left(u(\cdot,t)+\alpha\right)^{p}$}
\label{Section 3}
In this section 
we estimate the first term of $\Phi (t)$: 
\begin{align*}
\frac{1}{p} \int_{\Omega} 
\left(u(\cdot,t) 
+\alpha\right)^p. 
\end{align*} 
The following lemma gives 
an estimate for 
the derivative of the first term in $\Phi(t)$. 
\begin{lem}\label{ulemma}
	For $p \geq 1$, we have 
	\begin{align}\label{u}
		&\frac{1}{p} \frac{d}{dt}\int_{\Omega} 
		(u+\alpha)^{p}
		+
		\frac{p-1}{2}\int_{\Omega} (u+\alpha)^
		{p+m_1-3}|
		\nabla u|^2 
		\\\nonumber
		&\le \frac{{\chi}^2(p-1)}
		{2}\int_{\Omega} 
		(u+\alpha) ^{p+2m_2-m_1-3}
|\nabla v|^2.
	\end{align}
\end{lem}
\begin{proof}
	The first equation of \eqref{KS} 
	and integration by parts enable us to 
	see
	\begin{align*}
		\frac{1}{p}\dfrac{d}{dt}\int_{\Omega}(u+
\alpha)^p
		&=-\int_{\Omega}\nabla (u+\alpha)^{p-1} 
\cdot
		\left[(u+\alpha)^{m_1-1}\nabla u
		-\chi u(u+\alpha)^{m_2-2}
		\nabla v
\right]
		\\
		&=-(p-1)\int_{\Omega}(u+\alpha)^{p+m_1-3}
		|\nabla u|^2
		+\chi (p-1)\int_{\Omega}(u+\alpha)
		^{p+m_2-4}u\nabla u
		\cdot \nabla v
		\\
		&\leq -(p-1)\int_{\Omega}
		(u+\alpha)^{p+m_1-3}
		|\nabla u|^2+\chi(p-1)
		\int_{\Omega}(u+\alpha)
		^{p+m_2-3}|\nabla u\cdot\nabla v|.
	\end{align*}
By using Young's inequality, 
we obtain 
	\begin{align*}
		& \chi (p-1)(u+\alpha)^{p+m_2-3}
		|\nabla u\cdot \nabla v|
		\\
		&=\sqrt{p-1}(u+\alpha)^
		{\frac{p+m_1-3}{2}}
		|\nabla u|\cdot\chi\sqrt{p-1}
		(u+\alpha)^{\frac{p+2m_2-m_1-3}{2}}
		|\nabla v|
		\\
		&\leq\frac{p-1}{2}
		(u+\alpha)^{p+m_1-3}
		|\nabla u|^2
		+\frac{{\chi}^2(p-1)}{2}
		(u+\alpha)^{p+2m_2-m_1-3}
		|\nabla v|^2,
	\end{align*}
and hence 
we have
\begin{align*}
		&\frac{1}{p}\frac{d}{dt}\int_{\Omega}(u+
		\alpha)^p
		\\
		&\leq -\frac{p-1}{2}\int_{\Omega}
		(u+\alpha)^{p+m_1-3}|\nabla u|^2
		+\frac{{\chi}^2(p-1)}{2}
		\int_{\Omega}{(u+\alpha)}^
{p+2m_2-m_1-3
}|
		\nabla v|^2.
\end{align*}
	Therefore we can attain 
	the conclusion \eqref{ulemma}.
	\end{proof}
\section
{
An estimate for $\frac{1}{q}
\int_{\Omega}\left|\nabla v(\cdot,t)
\right|^{2q}
$
}
\label{Section 4}
In this section
we estimate the second term of $\Phi (t)$:
\begin{align*}
	\frac{1}{q}\int_{\Omega} |\nabla v(\cdot,t)|
	^{2q}.
\end{align*}
Although the following lemma is proved 
in a similar way as in the proof 
of the previous work (see 
\cite[Lemma 2.1]{Anderson-Deng_2017}), 
we shall reconstruct the method in 
\cite{Anderson-Deng_2017} and 
remove the convexity assumption. 
The following lemma presents 
an estimate 
for the derivative of the second term of 
$\Phi(t)$.
\begin{lem}\label{vlemma} 
	If $\delta \in (0,\frac{2(q-1)}{q^2})$, 
	then 
	there exists $D_{\delta}>0$ 
such that 
	\begin{align}
		\label{v}
		&\frac{1}{q}\frac{d}{dt} \int_{\Omega}
	 |\nabla v |^{2q}
		+ \left(\frac{2(q-1)}{q^2}
		-\delta\right)
		\int_{\Omega}\left|\nabla |\nabla v|^q
\right|^2 
		+ 2
		\int_{\Omega}|\nabla v|^{2q}
		\\\nonumber
		&\leq \frac{4(q-1)+n}{2}
	\int_{\Omega} 
		(u+\alpha)^2 |\nabla v|^{2q-2} 
		+ D_{\delta}
	\end{align}
for all q $\ge$ 1. 
\end{lem}
\begin{proof}
	We fix $\delta 
	\in (0,\frac{2(q-1)}{q^2})$.  
	Then we 
infer that
	\begin{align}\label{non-by-parts}
		\frac{1}{q}\frac{d}{dt}\int_{\Omega}|\nabla 
		v|^{2q}
		&=\int_{\Omega}|\nabla v|^{2(q-1)}
		\frac{\partial}{\partial t}|\nabla v|^2. 
	\end{align}
Due to \eqref{KS}, the second equation 
in \eqref{KS} entails us to see
	\begin{align}\label{pa-v}
		\frac{\partial}
		{\partial t}|\nabla v|^2
		&=2\nabla v
		\cdot\nabla v_t\\\nonumber
		&=2\nabla v\cdot \nabla[\Delta v-v+u]\\
\nonumber
		&=2\nabla v\cdot\nabla \Delta v
		-2|\nabla v|^2
		+2
\nabla u\cdot \nabla v.
\end{align}
Noticing from the chain rule that 
\begin{align*}
\Delta |\nabla v|^2 
&= 
\sum_{i=1}^n 
\frac{\partial^2}{\partial x_i^2} 
\left(
\sum_{j=1}^n\left(
\frac{\partial v}{\partial x_j} 
\right)^2\right)
\\
&=\sum_{i,j=1}^n
\frac{\partial }{\partial x_i} 
\left(2\frac{\partial v}{\partial x_j}
\cdot 
\frac{\partial^2 v}
{\partial x_i\partial x_j}\right)
\\
&= 2\left|D^2v\right|^2+2\nabla v \cdot 
\nabla 
\Delta v,  
\end{align*}
where $D^2 v$ denotes the 
Hessian matrix, 
we obtain
	\begin{align*}
		2\nabla v\cdot \nabla \Delta v
		=\Delta |\nabla v|^2-2|D^2v|. 
	\end{align*}
This together with 
\eqref{pa-v} yields 
	\begin{align*}
		\frac{\partial}{\partial t}
		|\nabla v|^2=\Delta|\nabla v|^2
		-2\left|D^2v\right|^2-2|\nabla v|^2
		+2\nabla u\cdot\nabla v.
	\end{align*}
Applying this identity to 
\eqref{non-by-parts}, we have 
	\begin{align}\label{v-parts}
		&\frac{1}{q}\dfrac{d}{dt}\int_{\Omega}
		|\nabla v|^{2q}
		+
		2\int_{\Omega}
		|\nabla v|^{2(q-1)}
		|D^2 v|^2+2\int_{\Omega}|\nabla v|^{2q}
		\\\nonumber
		&=\int_{\Omega}|\nabla v|^{2(q-1)}
		\Delta|\nabla v|^2
		+2\int_{\Omega}|\nabla v|^{2(q-1)}
\nabla u
		\cdot\nabla v.
	\end{align}
Here we 
see from integration by parts that 
	\begin{align*}
&\int_{\Omega}|\nabla v|^{2(q-1)}
\Delta |\nabla v|^2
=\int_{\partial\Omega}|\nabla v|^{2q-2
}
\nabla\left(|\nabla v|^2\right)\cdot \nu
-\int_{\Omega}\nabla\left(|\nabla v|
^{2(q-1)
}\right)
\cdot\nabla\left(|\nabla v|^2\right).
\end{align*}
	If $\delta 
	\in (0,\frac{2(q-1)}{q^2})$
	, then
	there exists $D_{\delta}>0$ such that
	\begin{align}\label{delta-estimate}
		\int_{\partial\Omega}|\nabla v|
		^{2q-2
}\nabla\left(|\nabla v|^2\right)\cdot\nu
		\leq\delta\int_{\Omega}
\left|\nabla|\nabla v|^q\right|
		^2+D_{\delta}
	\end{align}
(see Lemma \ref{non-convex}),  
and hence we rewrite 
\eqref{v-parts} as 
\begin{align}\label{v-all}
		&\frac{1}{q}\dfrac{d}{dt}\int_{\Omega}
		|\nabla v|^{2q}
		+\int_{\Omega}
		\nabla
\left(|\nabla v|^{2(q-1)
}\right)
		\cdot\nabla
\left(|\nabla v|^2\right)
		+2\int_{\Omega}|\nabla v
|^{2(q-1)}|
		D^2 v|^2
\\\nonumber
&\quad\,+2\int_{\Omega}|\nabla v|^{2q}
		\\\nonumber
		&\leq 2\int_{\Omega}|\nabla v|^{2(q-1)}
		\nabla u \cdot\nabla v
		+\delta\int_{\Omega}
\left|\nabla |\nabla v|^q\right|^2
		+D_{\delta}.
	\end{align}
Applying integration by parts 
to the first term on the right-hand side 
of \eqref{v-all} gives
\begin{align}\label{v0}
&2\int_{\Omega}|\nabla v|^{2(q-1)}
		\nabla u \cdot\nabla v
\\\nonumber
&= -2(q-1)\int_{\Omega}u|\nabla v|^{2(q-2)} 
\nabla
\left(|\nabla v|^2
\right)\cdot\nabla v
-
2\int_{\Omega}u|\nabla v|^{2(q-1)}
		\Delta v.
\end{align}
Now we estimate the following quantities: 
\begin{align*}
u|\nabla v|^{2(q-2)}\nabla\left(|\nabla v|
^2\right)
\cdot \nabla v, \quad 
u|\nabla v|^{2(q-1)}\Delta v.
\end{align*}
Using the inequality 
$u \le u+\alpha$, 
the pointwise inequality 
$|\Delta v|^2 \le n|D^2 v|^2$ and 
the Young inequality, 
we can 
notice that 
\begin{align}\label{v1}
-u|\nabla v|^{2(q-2)}\nabla\left(|\nabla v|
^2\right)
\cdot \nabla v
&\le 2 \cdot \frac{1}{2} 
|\nabla v|^{q-2}
\left|\nabla |\nabla v|^2\right|
\cdot (u+\alpha) |\nabla v|^{q-1}
\\\nonumber
&\le 
\frac{1}{4}\left(|\nabla v|^{q-2}
\left|\nabla|\nabla v|^2\right|\right)^2
+\left((u+\alpha)|\nabla v|^{q-1}\right)^2, 
\\\label{v2}
-u|\nabla v|^{2(q-1)}|\Delta v|
		&\le 
2 \cdot \sqrt{\frac{1}{n}}|\nabla v|
		^{q-1}
		\Delta v
		\cdot\sqrt{\frac{n}{4}}(u+\alpha)|\nabla 
		v|^{q-1}
		\\\nonumber
		&\le \frac{1}{n}
		|\nabla v|^{2(q-1)}|\Delta v|^2 
		+\frac{n}{4}
		\left((u+\alpha)|\nabla v|^{q-1}
		\right)^2 
\\\nonumber
&\le 
\int_{\Omega}|\nabla v|^{2(q-1)
}|D^2 v|^2 + 
\frac{n}{4}
		\left((u+\alpha)|\nabla v|^{q-1}
		\right)^2. 
	\end{align}
Applying \eqref{v1} and \eqref{v2} 
to \eqref{v0}, we obtain 
that
\begin{align}\label{v-last}
2\int_{\Omega}|\nabla v|^{2(q-1)}
		\nabla u \cdot\nabla v
&\le 
\frac{q-1}{2}\int_{\Omega}
|\nabla v|^{2(q-2)}\left|\nabla |\nabla v|
^2\right|^2
\\
\nonumber
&\quad\,+
\frac{4(q-1)+n}{2}
\int_{\Omega}
(u+\alpha)^2|\nabla v|^{2q-2}
\\\nonumber
&\quad\,
+2\int_{\Omega}|\nabla v|^{2(q-1)
}|D^2 v|^2.
\end{align}
From \eqref{v-all} and \eqref{v-last}, 
by using that 
	\begin{align*}
		&|\nabla v|^{2(q-2)}
		\left|\nabla|\nabla v|^2\right|^2
		=\frac{4}{q^2}\left|\nabla |\nabla v|^q
\right|^2 
	\end{align*}
as well as 
	\begin{align*}
		&\nabla (|\nabla v|^{2(q-1)})\cdot
		\nabla(|\nabla v|^2)
		=(q-1)|\nabla v|^{2(q-2)}
		\left|\nabla|\nabla v|^2\right|^2
	\end{align*}
for $q>1$, 
we can confirm that 
\begin{align*}
		&\frac{1}{q}\dfrac{d}{dt}\int_{\Omega}
		|\nabla v|^{2q}
		+
\frac{2(q-1)}{q^2}
		\int_{\Omega}\left|\nabla |\nabla v|^q
\right|^2
		+2\int_{\Omega}|\nabla v|^{2q}
		\\\nonumber
		&\le 
\frac{4(q-1)+n}{2}
\int_{\Omega}
(u+\alpha)^2|\nabla v|^{2q-2}
+\delta\int_{\Omega}\left|\nabla |\nabla v|^q
\right|^2
+D_{\delta}.
	\end{align*}
Thus we arrive at \eqref{vlemma}. 
\end{proof}
\begin{corollary}\label{vremark}
	If $\Omega$ is a convex bounded 
domain, 
then 
\eqref{v} is rewritten as 
	\begin{align*}
	&\frac{1}{q}\frac{d}{dt} \int_{\Omega} |
	\nabla v |^{2q}
	+\frac{2(q-1)}{q^2}\int_{\Omega}
\left|\nabla |
	\nabla v|^q\right|^2 
	+ 2\int_{\Omega}|\nabla v|^{2q}
	\\
	&\leq \frac{4(q-1)+n}{2}
	\int_{\Omega} (u+\alpha)^2 
	|\nabla v|^{2q-2} 
	\end{align*}
for all $q \ge 1$. 
In 
other words, 
$\delta$ and $D_{\delta}$ 
are taken as $0$ in Lemma \ref{vlemma}. 
\end{corollary}
\begin{proof}
Suppose that $\Omega$ 
	is a convex 
bounded 
domain. 
Then we see from Lemma \ref{convex} that 
	\begin{align*}
		\nabla \left(|\nabla v|^2
\right)\cdot \nu \leq 0, 
	\end{align*}
and so 
we can rewrite \eqref{delta-estimate} 
as 
\begin{align*}
\int_{\partial\Omega}|\nabla v|
		^{2q-2
}\nabla\left(|\nabla v|^2\right)\cdot\nu 
\le 0. 
\end{align*}
By 
an argument similar to 
the proof of Lemma 
\ref{vlemma}, 
we can attain the 
conclusion. 
\end{proof}
\section{Proof of the main theorem}
\label{Section 5}
In this section we prove 
Theorem \ref{main}. 
The following lemma plays 
an important role in the proof of Theorem 
\ref{KS}. 
\begin{lem}\label{PhiLemma}
Let $p, q\geq 1$,\ $p \neq -m_1+1$,\  
$\delta \in (0,\frac{2(q-1)}{q^2})$,\  
and let $\Phi(t)$ be 
defined as \eqref{def of phi}. 
Then  
there exists $D_{\delta}>0$ such that
	\begin{align}\label{Phi}
		&\frac{d\Phi}{dt}
		+\frac{p-1}{2}
\left(\frac{2}{p+m_1-1}
		\right)^2\int_{\Omega}\left|\nabla (u+
		\alpha)
		^{\frac{p+m_1-1}{2}}\right|^2
		+\left(\frac{2(q-1)}{q^2}-\delta 
		\right)\int_{\Omega}
\left|\nabla |\nabla v|^q\right|^2
		\\\nonumber
		&\leq
		\frac{{\chi}^2(p-1)}{2}\int_{\Omega} 
		(u+\alpha)^
		{
p+2m_2-m_1-3
}|\nabla v|^2
		+\frac{4(q-1)+n}{2}
		\int_{\Omega} (u+\alpha)^2 
		|\nabla v|^
		{2q-2}
		+D_{\delta}.
	\end{align}
\end{lem}
\begin{proof}
	The combination of \eqref{u}
	and \eqref{v} 
	yields \eqref{Phi}. 
In fact, due to \eqref{u}
	and \eqref{v}, we can find that 
for $p,q\ge1$, 
there exists $\delta \in (0,
	\frac{2(q-1)}{q^2})$ 
such that 
	\begin{align*}
		&
		\frac{1}{p} 
\frac{d}{dt}\int_{\Omega} 
		(u+\alpha)^{p}
		+\frac{1}{q} 
\frac{d}{dt}\int_{\Omega}
	 |\nabla v |^{2q}
		\\\nonumber
		&+
		\frac{p-1}{2}\int_{\Omega} (u+\alpha)^
		{p+m_1-3}|
		\nabla u|^2 
		+
		\left(\frac{2(q-1)}{q^2}-\delta\right)
		\int_{\Omega}|\nabla |\nabla v|^q|^2
		\\\nonumber
		&\le \frac{{\chi}^2(p-1)}
		{2}\int_{\Omega} 
		(u+\alpha) ^{
p+2m_2-m_1-3
}
|\nabla v|^2
		+
		\frac{4(q-1)+n}{2}\int_{\Omega} 
		(u+\alpha)^2 |\nabla v|^{2q-2} 
		+ D_{\delta}.  
	\end{align*}
	On the other hand, 
we notice that 
if $p \neq -m_1+1$, then 
	\begin{align*}
		\int_{\Omega}(u+\alpha)^{p+m_1-3}|\nabla u|
		^2=\left(\frac{2}{p+m_1-1}
		\right)^2\int_{\Omega}\left|\nabla (u+
		\alpha)
		^{\frac{p+m_1-1}{2}}\right|^2.
	\end{align*}
This together with \eqref{def of phi} 
clearly proves \eqref{PhiLemma}.	
\end{proof}
We are now in a position to complete 
the proof of Theorem \ref{main}. 
\begin{proof}
	[{\rm \bf Proof of Theorem \ref{main}}]
	\noindent
	We divide the proof  
	into four steps:
	\begin{enumerate}[\bf (Step 1):]	
		\item {
Estimates 
		for 
		$\int_{\Omega} \left(u+\alpha
\right)^{p+2m_2-m_1-3}
		|\nabla v|^2$\ and
		$\int_{\Omega} \left(u+\alpha
\right)^2 
		|\nabla v|^{2q-2}$. 
}
		\item {
Estimates for $
		\int_{\Omega}(u+
		\alpha)^{p\eta}$\ and
		$\int_{\Omega}
		|\nabla v|^{2q\eta}$. 
}
		\item {
Deriving an ordinary 
differential inequality 
		for $\Phi(t)$.  
		}
		\item {Establishing a lower bound 
		for the blow-up time $t^{\ast}$ 
in the classical sense.  
		}
	\end{enumerate}

\noindent
By means of these processes, 
we can argue a method to obtain 
a lower bound for the blow-up time 
for solutions of \eqref{KS}. 

\noindent\\
{\bf (Step 1)}
We  shall show that 
	\begin{align}\label{step1-1}
		&\int_{\Omega} (u+\alpha)
		^{p+2m_2-m_1-3}|\nabla v|^2
		\leq 
		\frac{1}{(q\eta)'}
		\left(\int_{\Omega}(u+\alpha)^{p\eta}
		\right )^{\frac{1}{\beta_1}}
		|\Omega|^{\frac{1}{{\beta_1}'}}
		+\frac{1}{q\eta}\int_{\Omega}
		|\nabla v|^{2q\eta},
		\\\label{step1-2}
		&\int_{\Omega} (u+\alpha)
		^{2}|\nabla v|^{2q-2}
		\leq \frac{1}{(q'\eta)'}
		\left(\int_{\Omega}(u+\alpha)^{p\eta}
		\right )^{\frac{1}{\beta_2}}
		|\Omega|^{\frac{1}{{\beta_2}'}}
		+ \frac{1}{q'\eta}\int_{\Omega}
		|\nabla v|^{2q\eta},
	\end{align}
	where $\eta$ is the constant defined 
as \eqref{df of eta} 
and $'$ denotes the H\"older conjugate 
exponent 
e.g., $q':=\frac{q}{q-1}$ 
and 
	\begin{align*}
		&{\beta_1}
:=\frac{p}{p+2m_2-m_1-3}\cdot
		\frac{q\eta-1}{q}>1,
		\\
		&{\beta_2}
		:=\frac{p}{2}\cdot\frac{q\eta-q+1}{q}>1. 
	\end{align*}
Applying Young's inequality to the first term 
on the left-hand side of \eqref{step1-1} 
gives 
	\begin{align}\label{step1-1-1}
		\int_{\Omega} (u+\alpha)
		^{p+2m_2-m_1-3}|\nabla v|^2
		\leq 
		\frac{1}{(q\eta)'}
		\int_{\Omega}(u+\alpha)^{(p+2m_2-m_1-3)
		(q\eta)'}
		+\frac{1}{q\eta}\int_{\Omega}
		|\nabla v|^{2q\eta}.
	\end{align}
Thanks to boundedness 
of $\Omega$, 
using H\"older's inequality, 
we have 
	\begin{align}\label{step1-1-2}
		\int_{\Omega}(u+\alpha)^{(p+2m_2-m_1-3)
		(q\eta)'}
		\leq 
		\left(\int_{\Omega}(u+\alpha)^{p\eta}
		\right )^{\frac{1}{\beta_1}}
		|\Omega|^{\frac{1}{{\beta_1}'}}, 
	\end{align}
where the condition \eqref{C2} enables us 
to take $\beta_1>1$ as 
\begin{align*}
{\beta_1}=\frac{p\eta}{(p+2m_2-m_1-3)
		(q\eta)'}=\frac{p}{p+2m_2-m_1-3}\cdot
		\frac{q\eta-1}{q}.
\end{align*}
Plugging \eqref{step1-1-2} 
into \eqref{step1-1-1}, 
we obtain \eqref{step1-1}.
	Similarly, 
combining Young's inequality with 
H\"older's inequality yields
	\begin{align*}
		\int_{\Omega} (u+\alpha)
		^{2}|\nabla v|^{2q-2}
		&\leq \frac{1}{(q'\eta)'}
		\int_{\Omega}(u+\alpha)^{2(q'\eta)'}
		+\frac{1}{q'\eta}\int_{\Omega}
		|\nabla v|^{2q\eta}
		\\
		\nonumber
		&\leq \frac{1}{(q'\eta)'}
		\left(\int_{\Omega}(u+\alpha)^{p\eta}
		\right )^{\frac{1}{\beta_2}}
		|\Omega|^{\frac{1}{{\beta_2}'}}
		+ \frac{1}{q'\eta}\int_{\Omega}
		|\nabla v|^{2q\eta}, 
	\end{align*}
where the condition \eqref{C2} enables us to 
take $\beta_2 >1$ as
	\begin{align*}
		&{\beta_2}=\frac{p\eta}{2(q'\eta)'}
		=\frac{p}{2}\cdot\frac{q\eta-q+1}{q}.
	\end{align*}
Therefore we arrive 
at \eqref{step1-2}. 

\noindent\\
	{\bf (Step 2)}
The purpose of this step is 
to obtain 
the following inequalities:  
	\begin{align}\label{step2-1}
		\int_{\Omega}{(u+\alpha)}^{p\eta}
		&\leq C_1(m_1) \ep \int_{\Omega}
		\left|\nabla (u+\alpha)^{\frac{p+m_1-1}
		{2}}\right|^2
		\\\nonumber
		&\quad\,+\frac{C_2(m_1)}{\ep^
		{\frac{1-ar\eta}{ar\eta}}}
		\left(\int_{\Omega}
		(u+\alpha)^{p}
		\right)^{f(\eta,r)}
		\hspace{-2mm}+
		C_3(m_1)\left(\int_{\Omega}
		(u+\alpha)^{p}
		\right)^{\eta}, 
		\\\label{step2-2}
		\int_{\Omega}|\nabla v|^{2q\eta}
		&\le C_4\ep\int_{\Omega}
		\left|\nabla|\nabla v|^q\right|^2
		+\frac{C_5}{\ep
		^{\frac{\eta}{2-\eta}}}
		\left(\int_{\Omega}
		|\nabla v|^{2q}\right)^{f(\eta,1)}
		+
		C_6\left(\int_{\Omega}
		|\nabla v|^{2q}\right)^{\eta}
	\end{align}	
	for all $\ep >0$, where $\eta$, $f$ and 
$r$ are defined as \eqref{df of eta}, 
\eqref{def of f} and 
\eqref{def of r}, respectively, and 
	\begin{align*}
		&C_1(m_1):=2^{2r\eta-1} ar\eta{c_1}^{2r
		\eta}
		, \quad
		C_2(m_1):=2^{2r\eta} 
(1-ar\eta){c_1}^{2r\eta}
		, \quad
		C_3(m_1):=2^{2r\eta-1} 
{c_1}^{2r\eta},
		\\
		&C_4:=2^{2\eta-1}\cdot
\frac{\eta}{2}\cdot{c_2}^{2\eta}
		, \quad
		C_5:=2^{2\eta-1}\cdot
\frac{2-\eta}{2}\cdot{c_2}^{2\eta}
		, \quad
		C_6:=2^{2\eta-1} {c_2}^{2\eta}.
	\end{align*}
In order to estimate 
$\int_{\Omega}(u+\alpha)^{p\eta}$\ and\  
$\int_{\Omega}|\nabla v|^{2q\eta}$ 
without wasting the power of diffusion, 
we apply Lemma \ref{G-N}. 
	Under the conditions \eqref{C1} and 
	\eqref{C2},
	we can show existence of $c_1>0$ such
	that
	\begin{align*}
		&\int_{\Omega}{(u+\alpha)}^{p\eta}
		\\
		&=\left\|(u+\alpha)^{\frac{p+m_1-1}{2}}
		\right\|_{L^{2r\eta}(\Omega)}^{2r\eta}
		\\
		&\leq {c_1}^{2r\eta}
		\left(
		{\left\|\nabla (u+\alpha)^{\frac{p+m_1-1}
		{2}}
		\right\|_{L^2(\Omega)}^a
		}
		{{\left\|(u+\alpha)^{\frac{p+m_1-1}{2}}
		\right\|}_{L^{2r}(\Omega)}^{1-a}
		}
		+{\left\|(u+\alpha)^{\frac{p+m_1-1}{2}}
		\right\|_
		{L^{2r}(\Omega)}
		}
		\right)^{2r\eta}
		\\
		&\leq 2^{2r\eta-1}{c_1}^{2r\eta}
		\left(
		{\left\|\nabla (u+\alpha)^{\frac{p+m_1-1}
		{2}}
		\right\|_{L^2(\Omega)}^{2ar\eta}
		}
		{{\left\|(u+\alpha)^{\frac{p+m_1-1}{2}}
		\right\|}_{L^{2r}(\Omega)}^
{2(1-a)r\eta
}
		}
		+{\left\|(u+\alpha)^{\frac{p+m_1-1}{2}}
		\right\|_
		{L^{2r}(\Omega)}^{2r\eta}
		}
		\right), 
	\end{align*}	
where 
\begin{align*}
a:=
\frac{
\frac{1}{2r}-\frac{1}{2r\eta}}
{\frac{1}{2r}+\frac{1}{n}-\frac{1}{2}
}\in(0,1). 
\end{align*}
Noting that 
$2ar\eta <2$ 
which implies the condition 
$p>\frac{\eta(m_1-1)}{(\eta-1)(\eta-2)}$ 
guaranteed by \eqref{C2}, 
	thanks to the Young inequality,
	we can estimate the first term 
on the right-hand side as
	\begin{align*}
		&{\left\|\nabla (u+\alpha)^{\frac{p+m_1-1}
		{2}}
		\right\|_{L^2(\Omega)}^{2ar\eta}
		}
		{{\left\|(u+\alpha)^{\frac{p+m_1-1}{2}}
		\right\|}_{L^{2r}(\Omega)}^
{2(1-a)r\eta
}
		}
		\\
		&
		\leq
		ar\eta\ep
		{\left\|\nabla(u+\alpha)^{\frac{p+m_1-1}
		{2}}
		\right\|_
		{L^{2}(\Omega)}^{2}
		}
		+
		\frac{1-ar\eta}{\ep^
		{\frac{1-ar\eta}{ar\eta}}}
		{\left\|(u+\alpha)^{\frac{p+m_1-1}{2}}
		\right\|_
		{L^{2r}(\Omega)}^
{
2r\frac{(1-a)\eta}
		{1-ar\eta}
}
		}
	\end{align*}
	for all $\ep>0$, 
and we see 
that 
$\frac{(1-a)\eta}{1-ar\eta}$ is rewritten as 
\begin{align*}
\frac{(1-a)\eta}{1-ar\eta}
&= \frac{\left(1-
\frac{\frac{1}{2r}-\frac{1}{2r\eta}}
{\frac{1}{2r}+\frac{1}{n}-\frac{1}{2}}
\right)\eta}{1-
\frac{\frac{1}{2r}-\frac{1}{2r\eta}}
{\frac{1}{2r}+
\frac{1}{n}-\frac{1}{2}}\cdot r\eta}
\\
&= 
\frac{\left(\left(\frac{1}{2r}+\frac{1}{n}-
\frac{1}{2}\right)-\left(
\frac{1}{2r}-\frac{1}{2r\eta}
\right)\right)\eta}
{\frac{1}{2r}+\frac{1}{n}-\frac{1}{2}
-\left(\frac{1}{2r}-\frac{1}{2r\eta}
\right)r\eta}
\\
&= 
1+ 
\frac{\eta-1} 
		{n\left(\frac{1}{n}-
\frac{\eta}{2}+\frac{1}
		{2r}\right)}
=f(\eta,r). 
\end{align*}
Thus 
we obtain
	\begin{align*}
		\int_{\Omega}{(u+\alpha)}^{p\eta}
		&\leq C_1(m_1) \ep \int_{\Omega}
		\left|\nabla (u+\alpha)^{\frac{p+m_1-1}
		{2}}\right|^2
		\\\nonumber
		&\quad\,+\frac{C_2(m_1)}{\ep^
		{\frac{1-ar\eta}{ar\eta}}}
		\left(\int_{\Omega}
		(u+\alpha)^{p}
		\right)^{f(\eta,r)}
		+
		C_3(m_1)\left(\int_{\Omega}
		(u+\alpha)^{p}
		\right)^{\eta},
	\end{align*}	
	with 
		$C_1(m_1)=2^{2r\eta-1} ar\eta{c_1}^{2r
		\eta}$, 
		$C_2(m_1)=2^{2r\eta-1} (1-ar\eta){c_1}
		^{2r
		\eta}$ and 
		$C_3(m_1)=2^{2r\eta-1} {c_1}^{2r\eta}$ 
	for 
all $\ep>0$. 
	In a similar way, there exists $c_2>0$ 
	such that
	\begin{align*}
\int_{\Omega}|\nabla v|^{2q\eta}
		&=\left\||\nabla v|^q
		\right\|_{L^{2\eta}(\Omega)}^{2\eta}
		\\
		\nonumber
		&\leq {c_2}^{2\eta}
		\left(
		{\left\|\nabla|\nabla v|^q
		\right\|_
		{L^2(\Omega)}^{\frac{1}{2}}
		}
		{\left\||\nabla v|^q\right\|_{L^2(\Omega)}
		^{\frac{1}{2}}
		}
		+{\left\||\nabla v|^q\right\|_
		{L^2(\Omega)}
		}
		\right)^{2\eta}
		\\
		\nonumber
		&\leq 2^{2\eta-1}{c_2}^{2\eta}
		\left(
		{\left\|\nabla|\nabla v|^q
		\right\|_
		{L^2(\Omega)}^{\eta}
		}
		{\left\||\nabla v|^q\right\|_{L^2(\Omega)}
		^{\eta}
		}
		+{\left\||\nabla v|^q\right\|_
		{L^2(\Omega)}^{2\eta}
		}
		\right)
		\\
		\nonumber
		&\leq 2^{2\eta-1}{c_2}^{2\eta}
		\left(
		\frac{\eta}{2} \ep
		{\left\|\nabla|\nabla v|^q
		\right\|_
		{L^2(\Omega)}^{2}
		}
		+
		\frac{2-\eta}{2\ep
		^{\frac{\eta}{2-\eta}}}
		{\left\||\nabla v|^q\right\|_{L^2(\Omega)}
		^{\frac{2\eta}{2-\eta}
}
		}
		+{\left\||\nabla v|^q\right\|_
		{L^2(\Omega)}^{2\eta}
		}
		\right)
		\\
		\nonumber
		&=:C_4\ep\int_{\Omega}
		\left|\nabla|\nabla v|^q\right|^2
		+\frac{C_5}{\ep
		^{\frac{\eta}{2-\eta}}}
		\left(\int_{\Omega}
		|\nabla v|^{2q}\right)^{f(\eta,1)}
		+
		C_6\left(\int_{\Omega}
		|\nabla v|^{2q}\right)^{\eta}
	\end{align*}
	for all $\ep >0$, where 
		$C_4=2^{2\eta-1} \cdot 
\frac{\eta}{2} \cdot {c_2}^{2\eta}$, 
		$C_5=2^{2\eta-1}\cdot
\frac{2-\eta}{2}\cdot{c_2}^{2\eta}$ and 
		$C_6=2^{2\eta-1} {c_2}^{2\eta}$. 
Hence 
we can obtain \eqref{step2-1} 
and \eqref{step2-2}. 
 
\noindent\\ 
{\bf (Step 3)} 
Plugging the results of Step $1$ and 
Step $2$ 
into \eqref{PhiLemma}, 
we shall show 
an ordinary 
differential inequality 
for $\Phi(t)$: 
	\begin{align}\label{d-i} 
		\frac{d\Phi}{dt}\leq A{\Phi} 
		^{f(\eta,r) 
}
		+B{\Phi}^{f(\eta,1)}+C{\Phi}^{\eta}+D.
	\end{align}
To this end we first 
deal with the first and 
second terms 
on the right-hand side 
of \eqref{PhiLemma}. 
Applying \eqref{step2-1} to 
$(\int_{\Omega}(u+\alpha)
^{p\eta})^{\frac{1}{\beta_1}}$ 
and $(\int_{\Omega}
(u+\alpha)^{p\eta})^{\frac{1}{\beta_2}}
$ 
appearing in 
\eqref{step1-1} and 
\eqref{step1-2} yields 
\begin{align}\label{step3-1}
		&\int_{\Omega} (u+\alpha)
		^{p+2m_2-m_1-3}|\nabla v|^2
		\leq 
		\frac{1}{(q\eta)'}
\cdot R^{\frac{1}{\beta_1}}
		|\Omega|^{\frac{1}{{\beta_1}'}}
		+\frac{1}{q\eta}\int_{\Omega}
		|\nabla v|^{2q\eta},
		\\
		\label{step3-2}
		&\int_{\Omega} (u+\alpha)
		^{2}|\nabla v|^{2q-2}
		\leq 
\frac{1}{(q'\eta)'}
\cdot R^{\frac{1}{\beta_2}}
		|\Omega|^{\frac{1}{{\beta_2}'}}
		+ \frac{1}{q'\eta}\int_{\Omega}
		|\nabla v|^{2q\eta}, 
	\end{align}
where $R$ is given by 
\begin{align}\label{Rdef}
R&:=C_1(m_1) \ep \int_{\Omega}
		\left|\nabla (u+\alpha)^{\frac{p+m_1-1}
		{2}}\right|^2
		\\\nonumber
		&
		\quad\,
		+\frac{C_2(m_1)}{\ep^
		{\frac{1-ar\eta}{ar\eta}}}
		\left(\int_{\Omega}
		(u+\alpha)^{p}
		\right)^{f(\eta,r)}
		\hspace{-2mm}+
		C_3(m_1)\left(\int_{\Omega}
		(u+\alpha)^{p}
		\right)^{\eta} 
\end{align}
with $\ep>0$ and 
$C_1(m_1), C_2(m_1)$ 
and $C_3(m_1)$ 
defined in Step 2. 
In order to show that 
$R^{\frac{1}{\beta_i}} \le R$ $(i=1,2)$, 
we shall show that $R \ge 1$. 
Indeed, 
focusing 
on the second term 
on the right-hand side of \eqref{Rdef}, 
by using 
the 
inequality $u+\alpha \ge \alpha >0$
and choosing $\ep$ small enough, we 
obtain 
	\begin{align*}
		\frac{C_2(m_1)
}{\ep^{\frac{1-ar\eta}
		{ar\eta}}}
		\left(\int_{\Omega} (u+\alpha)^p\right)
		^
		{f(\eta,r)}
		\geq 
\frac{C_2(m_1)}{\ep^
		{\frac{1-ar\eta}{ar\eta}}}
		\left(\alpha^p|\Omega|
		\right)^{f(\eta,r)}
\ge 1. 
		\end{align*}
Combining this inequality 
with \eqref{Rdef} 
entails that 
$R \ge 1$. 
Therefore we arrive at 
\[
R^{\frac{1}{\beta_i}} \le R \quad 
(i=1,2). 
\] 
Plugging this inequality 
into \eqref{step3-1} and \eqref{step3-2}, 
we obtain 
\begin{align*}
		&\int_{\Omega} (u+\alpha)
		^{p+2m_2-m_1-3}|\nabla v|^2
		\leq 
		\frac{1}{(q\eta)'}\cdot
		R 
		|\Omega|^{\frac{1}{{\beta_1}'}}
		+\frac{1}{q\eta}\int_{\Omega}
		|\nabla v|^{2q\eta},
		\\
		&\int_{\Omega} (u+\alpha)
		^{2}|\nabla v|^{2q-2}
		\leq 
		\frac{1}{(q'\eta)'}\cdot
		R 
		|\Omega|^{\frac{1}{{\beta_2}'}}
		+ \frac{1}{q'\eta}\int_{\Omega}
		|\nabla v|^{2q\eta}.
	\end{align*}
By applying these two inequalities   
to \eqref{PhiLemma}, 
we see 
\begin{align}\label{com}
		&\frac{d\Phi}{dt}
		+\frac{p-1}{2}{\left(\frac{2}{p+m_1-1}
		\right)}^2\int_{\Omega}\left|\nabla 
		(u+\alpha)^
		{\frac{p+m_1-1}{2}}\right|^2
\\\nonumber
		&\quad\,+\left(\frac{2(q-1)}{q^2}-\delta 
		\right)\int_{\Omega}
\left|\nabla \left|\nabla v\right|^q\right|^2
		\\\nonumber
		&\leq
		\frac{{\chi}^2(p-1)}{2}
		\left(
		\frac{1}{(q\eta)'} \cdot R 
		|\Omega|^{\frac{1}{{\beta_1}'}}
		+\frac{1}{q\eta}\int_{\Omega}
		|\nabla v|^{2q\eta}
		\right)
		\\\nonumber
		&\quad\,+\frac{4(q-1)+n}{2}
		\left(
		\frac{1}{(q'\eta)'} \cdot
		R 
		|\Omega|^{\frac{1}{{\beta_2}'}}
		+ \frac{1}{q'\eta}\int_{\Omega}
		|\nabla v|^{2q\eta}
		\right)
		+D_{\delta}
		\\\nonumber
		&=
E_1R+E_2\int_{\Omega}|\nabla v|^{2q\eta}
+D_\delta,
	\end{align}
where
\begin{align*}
E_1
&:=
		\frac{{\chi}^2(p-1)}{2}
		\cdot
		\frac{1}{(q\eta)'}
\cdot
		|\Omega|^{\frac{1}
		{{\beta_1}'}}
		+
		\frac{4(q-1)+n}{2}
		\cdot
		\frac{1}{(q'\eta)'}
\cdot
		|\Omega|^{\frac{1}
		{{\beta_2}'}}, 
\\
E_2
&:=
		\frac{{\chi}^2(p-1)}{2}
		\cdot
		\frac{1}{q\eta}
		+
		\frac{4(q-1)+n}{2}
		\cdot
		\frac{1}{q'\eta}.
\end{align*}
Plugging \eqref{step2-2} 
into \eqref{com}, 
we can rearrange 
\eqref{PhiLemma} 
as follows: 	
\begin{align*}
		&\frac{d\Phi}{dt}
		+\frac{p-1}{2}{\left(\frac{2}{p+m_1-1}
		\right)}^2\int_{\Omega}\left|\nabla 
		(u+\alpha)^
		{\frac{p+m_1-1}{2}}\right|^2
		+\left(\frac{2(q-1)}{q^2}-\delta 
		\right)\int_{\Omega}|\nabla |\nabla v|^q|^2
		\\\nonumber
		&\le 
E_1R+
E_2\left(
C_4\ep\int_{\Omega}
		\left|\nabla|\nabla v|^q\right|^2
		+\frac{C_5}{\ep
		^{\frac{\eta}{2-\eta}}}
		\left(\int_{\Omega}
		|\nabla v|^{2q}\right)^{f(\eta,1)}
		+
		C_6\left(\int_{\Omega}
		|\nabla v|^{2q}\right)^{\eta}
\right)
+D_\delta. 
	\end{align*}
Recalling the definition of $R$\ 
(see \eqref{Rdef}), 
we infer 
\begin{align}\label{nPhi}
		&\frac{d\Phi}{dt}
		+\left(\frac{p-1}{2}{\left(\frac{2}
		{p
		+m_1-1}
		\right)}^2
		-
		E_1C_1(m_1)\ep
		\right)
		\int_{\Omega}\left|\nabla 
		(u+\alpha)^
		{\frac{p+m_1-1}{2}}\right|^2
		\\\nonumber
		&\quad\,+
		\left(
		\left(\frac{2(q-1)}{q^2}-\delta 
		\right)	
		-
		E_2C_4\ep
		\right)
		\int_{\Omega}|\nabla |\nabla v|^q|^2
		\\\nonumber
		&\leq
		E_1C_2(m_1)\ep^{-\frac{1-ar\eta}{ar\eta}}
		\left(
		\int_{\Omega}(u+\alpha)^p
		\right)^{f(\eta,r)}
		+
		E_2C_5\ep^{-\frac{\eta}{2-\eta}}
		\left(
		\int_{\Omega} |\nabla v|^{2q}
		\right)^{f(\eta,1)}
		\\\nonumber
		&\quad\,+ 
		E_1C_3(m_1)
		\left(
		\int_{\Omega}(u+\alpha)^p
		\right)^{\eta}
		+
		E_2C_6
		\left(
		\int_{\Omega} |\nabla v|^{2q}
		\right)^{\eta}
		+D_{\delta}. 
	\end{align}
Then it follows that 
the second and third terms 
on the left-hand side 
of \eqref{nPhi} are nonnegative. 
We now fix 
$\delta \in (0,\frac{2(q-1)}{q^2})$  
and choose 
	$\ep$ small enough to satisfy not only 
\eqref{com} but also 
	\begin{align*}
		E_1C_1(m_1)\ep
		\leq \frac{p-1}{2}
		{\left(\frac{2}{p+m_1-1}\right)}^2, 
\quad 
		E_2C_2(m_1)\ep
		\leq \frac{2(p-1)}{q^2}-\delta.
	\end{align*}
Therefore 
we can rewrite \eqref{nPhi} as 
\begin{align*}
		\frac{d\Phi}{dt}
		&\leq
		E_1C_2(m_1)\ep^{-\frac{1-ar\eta}{ar\eta}}
		\left(
		\int_{\Omega}(u+\alpha)^p
		\right)^{f(\eta,r)}
		+
		E_2C_5\ep^{-\frac{\eta}{2-\eta}}
		\left(
		\int_{\Omega} |\nabla v|^{2q}
		\right)^{f(\eta,1)}
		\\\nonumber
		&\quad\,+ 
		E_1C_3(m_1)
		\left(
		\int_{\Omega}(u+\alpha)^p
		\right)^{\eta}
		+
		E_2C_6
		\left(
		\int_{\Omega} |\nabla v|^{2q}
		\right)^{\eta}
		+D_{\delta}.
	\end{align*}
Noting that 
		\begin{align*}
			\int_{\Omega} 
			(u+\alpha)^p\leq p 
			\Phi(t),
			\quad
			\int_{\Omega} 
			|\nabla v|^{2q} \leq q\Phi(t), 
		\end{align*}
we can establish \eqref{d-i} with 
\begin{align*}
A&=A(m_1)
:=p^{f(\eta,r)}
E_1%
C_2(m_1
) 
{\ep}^{-\frac{1-ar\eta}{ar
			\eta}},
\\
B&:=
q^{f(\eta,1)}
E_2
C_5
{\ep}^{-\frac{\eta}{2-\eta}},
\\
C&=C(m_1)
:=p^{\eta}
E_1
C_3(m_1) + q^{\eta} E_2 C_6,
\\
D&:=D_{\delta}.
\end{align*}	
Thus 
\eqref{d-i} holds. 

\noindent\\ 
	{\bf (Step 4)}
In this step we establish 
the following lower bound 
for the blow-up time $t^{\ast}$ 
for solutions 
of \eqref{KS} 
in the classical sense: 
\begin{align}\label{l}
		t^{\ast} \geq 
		\int_{\Phi(0)}^{\infty} 
		\frac{d\tau} 
		{A{\tau}^{f(\eta,r)} 
		+ 
		B{\tau}^{f(\eta,1)} 
		+ 
		C\tau^{\eta}
		 + 
		D}. 
	\end{align}
We first show 
that we can estimate 
a lower bound for the blow-up time 
in $\Phi$-measure. 
Indeed, we put 
\begin{align*}
G(\Phi(t)):=A{\Phi(t)}
		^{f(\eta,r)}
		+B{\Phi(t)}^{f(\eta,1)}+C{\Phi(t)^\eta}+D
\end{align*} 
	and
	\begin{align}\nonumber
		H(x):=\int_{\Phi(0)}^{x}\frac{d\tau}
		{G(\tau)} \quad (x \ge 0). 
	\end{align}
Since 
$f(\eta,s)>1\ (s>0)$, we notice that 
$\lim_{x \nearrow \infty} 
H(x)$ exists, 
and hence 
we obtain 
from 
the chain rule 
and the inequality 
$\frac{d\Phi(t)}{dt} \leq G(\Phi(t))$ 
(see \eqref{d-i})
that 
	\begin{align*}
		\frac{d}{dt}\left[H(\Phi(t))\right]
		&=
		\frac{1}{G(\Phi(t))} 
\cdot 
\frac{d\Phi(t)}{dt} \leq 1.
	\end{align*}
%
By integrating 
from $0$ to $t_{p,q}^{\ast}$, we have 
\begin{align*}
H\left(\Phi(t_{p,q}^{\ast})\right) - 
H\left(\Phi(0)\right) 
\le 
t_{p,q}^{\ast}. 
\end{align*}
Noting that 
$\lim_{t \nearrow {t_{p,q}}^{\ast}} 
\Phi(t) = \infty$ 
and $H(\Phi(0))=0$, 
we can attain 
that 
\begin{align}\label{before-time}
		t_{p,q}^{\ast} \geq 
		\int_{\Phi(0)}^{\infty} 
		\frac{d\tau} 
		{A{\tau}^{f(\eta,r)} 
		+ 
		B{\tau}^{f(\eta,1)} 
		+ 
		C\tau^{\eta}
		 + 
		D}. 
	\end{align}
Furthermore, 
we can 
regard the blow-up time for 
solutions of \eqref{KS} 
in $\Phi$-measure 
as that	 in the classical sense 
under the condition \eqref{C1}, 
i.e., 
\begin{align}\label{befor-after}
t_{p,q}^{\ast}=t^{\ast} 
\end{align}
(see Corollary \ref{non-gap}). 
A combination of 
\eqref{before-time} with 
\eqref{befor-after} 
yields 
\begin{align*}
t^{\ast}=t_{p,q}^{\ast} \geq 
		\int_{\Phi(0)}^{\infty} 
		\frac{d\tau} 
		{A{\tau}^{f(\eta,r)} 
		+ 
		B{\tau}^{f(\eta,1)} 
		+ 
		C\tau^{\eta}
		 + 
		D}. 
\end{align*}
Thus we arrive at \eqref{l}. 
In 
conclusion, 
the proof of Theorem \ref{main} is 
completed.
\end{proof} 
\smallskip
\section*{Acknowledgments}
The authors thank  
Dr.\ Masaaki Mizukami for his advice, 
especially for informing them 
a recent result by
Freitag \cite{Freitag_2018}. 
Thanks to his advice,
the authors could eliminate a gap
between the blow-up time 
for solutions of \eqref{KS} 
in $\Phi$-measure 
and that in the classical sense. 
\\
\bibliographystyle{plain}

\begin{thebibliography}{10}

\bibitem{Anderson-Deng_2017}
J.~R. Anderson and K.~Deng.
\newblock A lower bound on the blow up time 
for solutions of a chemotaxis
  system with nonlinear chemotactic 
  sensitivity.
\newblock {\em Nonlinear Anal.}, 159:2--9, 
2017.

\bibitem{Bellomo-Bellouquid-Tao-Winkler_2015}
N.~Bellomo, A.~Bellouquid, Y.~Tao, and 
M.~Winkler.
\newblock Toward a mathematical theory of 
{K}eller--{S}egel models of pattern
  formation in biological tissues.
\newblock {\em Math.\ Models Methods Appl.\ 
Sci.}, 25(9):1663--1763, 2015.

\bibitem{Cieslak-Stinner_2012}
T.~Cie\'{s}lak and C.~Stinner.
\newblock Finite-time blowup and 
global-in-time unbounded solutions to a
  parabolic--parabolic quasilinear 
  {K}eller--{S}egel system in higher
  dimensions.
\newblock {\em J. Differential Equations}, 
252(10):5832--5851, 2012.

\bibitem{Cieslak-Stinner_2014}
T.~Cie\'{s}lak and C.~Stinner.
\newblock Finite-time blowup in a 
supercritical quasilinear
  parabolic--parabolic {K}eller--{S}egel 
  system in dimension 2.
\newblock {\em Acta Appl.\ Math.}, 
129:135--146, 2014.

\bibitem{Enache_2011}
C.~Enache.
\newblock Lower bounds for blow-up time in 
some non-linear parabolic problems
  under {N}eumann boundary conditions.
\newblock {\em Glasg.\ Math.\ J.}, 
53(3):569--575, 2011.

\bibitem{Ferina-Marras-Vigilialoro_2015}
M.~A. Farina, M.~Marras, and G.~Viglialoro.
\newblock On explicit lower bounds and blow-
up 
times in a model of chemotaxis.
\newblock {\em Discrete Contin.\ Dyn.\ Syst.}, 
(Dynamical systems, differential
  equations and applications. 10th AIMS 
  Conference. Suppl.):409--417, 2015.

\bibitem{Freitag_2018}
M.~Freitag.
\newblock Blow-up profiles and refined 
extensibility criteria in quasilinear
  {K}eller--{S}egel systems.
\newblock {\em J. Math.\ Anal.\ Appl.}, 
463(2):964--988, 2018.

\bibitem{Hashira-Ishida-Yokota_2018}
T.~Hashira, S.~Ishida, and T.~Yokota.
\newblock Finite-time blow-up for quasilinear 
degenerate {K}eller--{S}egel
  systems of parabolic--parabolic type.
\newblock {\em J. Differential Equations}, 
264(10):6459--6485, 2018.

\bibitem{Hillen-Painter_2009}
T.~Hillen and K.~J. Painter.
\newblock A user's guide to {PDE} models for 
chemotaxis.
\newblock {\em J. Math.\ Biol.}, 
58(1-2):183--217, 2009.

\bibitem{Horstmann_2003}
D.~Horstmann.
\newblock From 1970 until present: the 
{K}eller--{S}egel model in chemotaxis
  and its consequences. {I}.
\newblock {\em Jahresber.\ Deutsch.\ 
Math.--Verein.}, 105(3):103--165, 2003.

\bibitem{Horstmann_2004}
D.~Horstmann.
\newblock From 1970 until present: the 
{K}eller--{S}egel model in chemotaxis
  and its consequences. {II}.
\newblock {\em Jahresber.\ Deutsch.\ 
Math.--Verein.}, 106(2):51--69, 2004.

\bibitem{Horstmann-Winkler_2005}
D.~Horstmann and M.~Winkler.
\newblock Boundedness vs. blow-up in a 
chemotaxis system.
\newblock {\em J. Differential Equations}, 
215(1):52--107, 2005.

\bibitem{Ishida-Seki-Yokota_2014}
S.~Ishida, K.~Seki, and T.~Yokota.
\newblock Boundedness in quasilinear 
{K}eller--{S}egel systems of
  parabolic--parabolic type on non-convex 
  bounded domains.
\newblock {\em J. Differential Equations}, 
256(8):2993--3010, 2014.

\bibitem{Ishida-Yokota_2012-2}
S.~Ishida and T.~Yokota.
\newblock Global existence of weak solutions 
to quasilinear degenerate
  {K}eller--{S}egel systems of 
  parabolic--parabolic type.
\newblock {\em J. Differential Equations}, 
252(2):1421--1440, 2012.

\bibitem{Ishida-Yokota_2012}
S.~Ishida and T.~Yokota.
\newblock Global existence of weak solutions 
to quasilinear degenerate
  {K}eller--{S}egel systems of 
  parabolic--parabolic type with small data.
\newblock {\em J. Differential Equations}, 
252(3):2469--2491, 2012. 

\bibitem{Ishida-Yokota_2020}
S.~Ishida and T.~Yokota.
\newblock Boundedness in a quasilinear fully 
parabolic {K}eller--{S}egel system
  via maximal {S}obolev regularity.
\newblock {\em Discrete Contin.\ Dyn.\ Syst.\ 
Ser.\ S}, 13(2):211--232, 2020. 

\bibitem{Jiao-Zeng_2018}
Y.~Jiao and W.~Zeng.
\newblock Blow-up and delay for a 
parabolic--elliptic {K}eller--{S}egel system
  with a source term.
\newblock {\em Bound.\ Value Probl.}, 
Paper No. 95, 10, 2018. 

\bibitem{Keller-Segel_1970}
E.~F. Keller and L.~A. Segel.
\newblock Initiation of slime mold 
aggregation 
viewed as an instability.
\newblock {\em J. Theoret.\ Biol.}, 
26:399--415, 1970.

\bibitem{Li-Zheng_2013}
J.~Li and S.~Zheng.
\newblock A lower bound for blow-up time in a 
fully parabolic {K}eller--{S}egel
  system.
\newblock {\em Appl.\ Math.\ Lett.}, 
26(4):510--514, 2013.

\bibitem{Li-Lankeit_2016}
Y.~Li and J.~Lankeit.
\newblock Boundedness in a 
chemotaxis-haptotaxis model with nonlinear
  diffusion.
\newblock {\em Nonlinearity}, 
29(6):1564--1595, 2016.

\bibitem{Marras-Piro-Viglialoro_2015}
M.~Marras, S.~Vernier~Piro, and G.~Viglialoro.
\newblock Lower bounds for blow-up in a 
parabolic--parabolic {K}eller--{S}egel
  system.
\newblock {\em Discrete Contin. Dyn. Syst.}, 
(Dynamical systems, differential
  equations and applications. 10th AIMS 
  Conference. Suppl.):809--816, 2015.

\bibitem{Marras-Piro-Viglialoro_2016}
M.~Marras, S.~Vernier-Piro, and G.~Viglialoro.
\newblock Blow-up phenomena in chemotaxis 
systems with a source term.
\newblock {\em Math.\ Methods Appl.\ Sci.}, 
39(11):2787--2798, 2016.

\bibitem{Mimura_2017}
Y.~Mimura.
\newblock The variational formulation of the 
fully parabolic {K}eller--{S}egel
  system with degenerate diffusion.
\newblock {\em J. Differential Equations}, 
263(2):1477--1521, 2017.

\bibitem{Mizoguchi-Winkler}
N.~Mizoguchi and M.~Winkler.
\newblock Blow-up in the two-dimensional 
parabolic {K}eller--{S}egel system.
\newblock preprint.

\bibitem{Painter-Hillen_2002}
K.~J. Painter and T.~Hillen.
\newblock Volume-filling and quorum-sensing 
in 
models for chemosensitive
  movement.
\newblock {\em Can.\ Appl.\ Math.\ Q.}, 
10(4):501--543, 2002.

\bibitem{Payne-Philippin-Piro_2010}
L.~E. Payne, G.~A. Philippin, and 
S.~Vernier~Piro.
\newblock Blow-up phenomena for a semilinear 
heat equation with nonlinear
  boundary condition, {II}.
\newblock {\em Nonlinear Anal.}, 
73(4):971--978, 2010.

\bibitem{Payne-Schaefer_2006}
L.~E. Payne and P.~W. Schaefer.
\newblock Lower bounds for blow-up time in 
parabolic problems under {N}eumann
  conditions.
\newblock {\em Appl.\ Anal.}, 
85(10):1301--1311, 2006.

\bibitem{Payne-Song_2010}
L.~E. Payne and J.~C. Song.
\newblock Blow-up and decay criteria for a 
model of chemotaxis.
\newblock {\em J. Math.\ Anal.\ Appl.}, 
367(1):1--6, 2010.

\bibitem{Payne-Song_2012}
L.~E. Payne and J.~C. Song.
\newblock Lower bounds for blow-up in a model 
of chemotaxis.
\newblock {\em J. Math.\ Anal.\ Appl.}, 
385(2):672--676, 2012.

\bibitem{Philippin_2015}
G.~A. Philippin.
\newblock Lower bounds for blow-up time in a 
class of nonlinear wave equations.
\newblock {\em Z. Angew.\ Math.\ Phys.}, 
66(1):129--134, 2015.

\bibitem{Senba-Suzuki_2006}
T.~Senba and T.~Suzuki.
\newblock A quasi-linear parabolic system of 
chemotaxis.
\newblock {\em Abstr.\ Appl.\ Anal.}, 2006.

\bibitem{Tao-Piro_2016}
Y.~Tao and S.~Vernier~Piro.
\newblock Explicit lower bound of blow-up 
time 
in a fully parabolic chemotaxis
  system with nonlinear cross-diffusion.
\newblock {\em J. Math.\ Anal.\ Appl.}, 
436(1):16--28, 2016.

\bibitem{Tao-Winkler_2012}
Y.~Tao and M.~Winkler.
\newblock Boundedness in a quasilinear 
parabolic--parabolic {K}eller--{S}egel
  system with subcritical sensitivity.
\newblock {\em J.\ Differential Equations}, 
252(1):692--715, 2012.

\bibitem{Winkler_2013}
M.~Winkler.
\newblock Finite-time blow-up in the 
higher-dimensional parabolic--parabolic
  {K}eller--{S}egel system.
\newblock {\em J. Math.\ Pures Appl.\ (9)}, 
100(5):748--767, 2013.

\end{thebibliography}

\end{document}